\documentclass[reqno,11pt, letterpaper]{amsart}
\PassOptionsToPackage{fleqn}{amsmath}
\RequirePackage{amsthm}
\RequirePackage{amsmath}
\RequirePackage{latexsym}
\RequirePackage{amssymb}
\RequirePackage{amscd}
\RequirePackage{epsfig}
\RequirePackage{graphics}
\RequirePackage{ifthen}
\RequirePackage{varioref}
\usepackage{changes}
\usepackage{soul}
\usepackage{xcolor}
\usepackage{color}
\usepackage[misc]{ifsym}
\usepackage{bm}
\usepackage{lipsum}

\usepackage{epstopdf}

\definecolor{darkgreen}{rgb}{0.0625,0.64,0.0625}
\ifpdf
  \usepackage[
    pdftex,
    colorlinks,%
    linkcolor=blue,citecolor=blue,urlcolor=blue,
    hyperindex,%
    plainpages=false,%
    bookmarksopen,%
    bookmarksnumbered%
  ]{hyperref}
\else
  \usepackage{hyperref}
\fi
\usepackage{bookmark}
\allowdisplaybreaks[1]

\def\R{{\mathbb R}}

\theoremstyle{plain}
\newtheorem{thm}{Theorem}[section]
\newtheorem{lem}[thm]{Lemma}
\newtheorem{prop}[thm]{Proposition}
\newtheorem{cor}[thm]{Corollary}
\theoremstyle{definition}
\newtheorem{rem}[thm]{Remark}
\newtheorem{exmp}{Example}[section]
\newtheorem{defn}[thm]{Definition}

\def\eq#1{{\rm(\ref{#1})}}
\def\ro#1{{\rm #1}}

\def\Bbb#1{{\mathbb#1}}

\def\R{\Bbb R}

\def\Rx{\R\mkern1mu^}

\def\Rn{\Rx n}

\def\SA#1{\ro{SA}(#1)} 
\def\SAR#1{\SA{#1,\R}} 
\def\GL#1{\ro{GL}(#1)} 
\def\GLR#1{\GL{#1,\R}} 
\def\SL#1{\ro{SL}(#1)} 
\def\SLR#1{\SL{#1,\R}}  

\def\semidirect{\ltimes}

\def\ip#1#2{\langle \,{#1}\,,{#2}\,\rangle}

\numberwithin{equation}{section}

\begin{document}

\title[The maximal curves and heat flow in fully affine geometry]{The maximal curves and heat flow \\in fully affine geometry}

    \author[Y. Yang]{Yun Yang}
    \address{ Yun Yang\newline\indent
     Department of Mathematics, Northeastern University, Shenyang, 110819, P.R. China}
    \email{yangyun@mail.neu.edu.cn}

    \begin{abstract}
       In Euclidean geometry, the shortest distance between two points is a {\it straight line}.
       Chern made a conjecture (cf. \cite{che}) in 1977  that an affine maximal graph of a smooth,
       locally uniformly convex function on two dimensional Euclidean space $\R^2$ must be a {\it paraboloid}.
       In 2000, Trudinger and Wang completed the proof of this conjecture in affine geometry (cf. \cite{tw-1}).
      ({\it Caution: in these literatures,
      the term ``affine geometry'' refers to ``equi-affine geometry''}.)
      A natural problem arises: Whether the {\it hyperbola} is a fully affine maximal curve in $\R^2$?
      In this paper, by utilizing the evolution equations for curves, we obtain the second variational formula for fully affine extremal curves in $\R^2$,
      and show the fully affine maximal curves in $\R^2$ are much more abundant and include the explicit curves $y=x^\alpha ~\left(\alpha\;\text{is a constant and}\;\alpha\notin\{0,1,\frac{1}{2},2\}\right)$ and $y=x\log x$.
      At the same time, we generalize the fundamental theory of curves in higher dimensions, equipped with $\text{GA}(n)=\text{GL}(n)\ltimes\R^n$.
      Moreover, in fully affine plane geometry, an isoperimetric inequality is investigated, and a complete classification of the solitons for fully affine heat flow is provided.
      We also study the local existence, uniqueness, and long-term behavior of this fully affine heat flow.
      A closed embedded curve will converge to an ellipse when evolving according to the fully affine heat flow is proved.
    \end{abstract}

\maketitle


\section{Introduction}
    Following the general spirit in the Erlangen program of Klein, fully affine differential geometry is based on the Lie group $\text{GA}(n,\R) = \GLR n \semidirect \Rn$
    consisting of affine transformations $x \longmapsto Ax+b$, $A\in \GLR n$, $b\in \Rn$ acting on $x \in \Rn$.
    In this geometry a key issue of study is the resulting invariants associated with submanifolds $M \subset \Rn$ (see Nomizu and Sasaki \cite{ns} and Simon \cite{sim} for details).
    Note that, in most of the literatures , ``affine geometry'' actually means ``equi-affine geometry'',
    in which one restricts to the subgroup $\SAR n = \SLR n \semidirect \Rn$ of volume-preserving affine transformations.
    In this paper, we restrict our attention to fully affine differential geometry. The main points of discussion and conclusions can be highlighted in four parts:
    (1) fully affine maximal curves, (2) fully affine isoperimetric inequality, (3) fully affine arc length in higher dimensions, (4)  heat flow in fully affine geometry.

\subsection{Fully affine maximal curves}
   Around 1977 Chern \cite{che} conjectured that an affine maximal graph of a smooth,
   locally uniformly convex function on two dimensional Euclidean space, $\R^2$, must be a paraboloid.
   Trudinger and Wang  \cite{tw-1} proved the validity of this conjecture in affine geometry and showed the corresponding result holds in higher dimensions.
   Furthermore, they studied the Plateau problem for affine maximal hypersurfaces,
   which is the affine invariant analogue of the classical Plateau problem for minimal
   surfaces \cite{tw-2}. Wang stated in \cite{wan} the concept of affine maximal surface in affine geometry corresponds to that of minimal
   surface in Euclidean geometry (Calabi \cite{cal} suggested using the terminology affine maximal as the
   second variation of the affine area functional is negative).
   The affine Bernstein problem and affine Plateau problem,
   as proposed in \cite{cal,cal-1,che}, are two fundamental problems for affine maximal submanifolds.

    Recently Kobayashi and Sasaki \cite{ks} investigated the first variational formula and found some extremal curves on plane with the fully affine group $\text{GA}(2)=\text{GL}(2)\ltimes\R^2$.
    However, the second variational formula is still a remaining problem.
    In the present paper, by employing the evolution equations of curve flow, we obtain the second variational formula
    and further study the local stability of  the extremal curves with respect to fully affine arc length (see Section \ref{sec-VF} for details).
    It is noteworthy that the fully affine maximal curves in $\R^2$ produce a more abundant category, which contains not only the explicit curves
    $\displaystyle y=x^\alpha$ ($\displaystyle\alpha\;\text{is a constant and}\;\alpha\notin\{0,1,\frac{1}{2},2\}$) and $\displaystyle y=x\log x$,
    but also these  implicit curves with the fully affine curvature $\displaystyle \varphi=\frac{3\sqrt{2}}{2}\tanh\left(\frac{\sqrt{2}}{3}\xi\right)$, $\displaystyle\varphi=\frac{3\sqrt{2}}{2}\coth\left(\frac{\sqrt{2}}{3}\xi\right)$,
    or $\displaystyle\varphi=\pm\frac{\sqrt{2}}{2}+\frac{9}{2\xi}$, where $\xi$ is the fully affine arc length parameter.

\subsection{Fully affine isoperimetric inequality}
    The solution to the classical isoperimetric problem can be represented in the form of an inequality which usually relates the length $L$ of a closed curve
    and the area $A$ of the planar region that it encloses. On Euclidean plane, the isoperimetric inequality is expressed by
    \begin{equation*}
    L^2\geq 4\pi A,
    \end{equation*}
    and that the equality holds if and only if the curve is a circle.
    In equi-affine geometry, the equi-affine isoperimetric inequality related all  ovals with an area $A$ and equi-affine perimeter $\bar{L}$ , is
    (cf. \cite{su})
    \begin{equation}\label{su-isoi}
       \bar{L}^3\leq8\pi^2A,
    \end{equation}
    and equality holds only for the ellipse.
    Note that the inequality sign flips, which is a little similar to the concept of affine maximal
    in affine geometry corresponding to that of minimal in
    Euclidean geometry.
    In addition, Sapiro and Tannenbaum \cite{st} presented an equi-affine isoperimetric inequality formed as
    \begin{equation}\label{st-isoi}
       2\oint\mu d\sigma\leq\frac{\bar{L}^2}{A},
    \end{equation}
    where $\mu$ is equi-affine curvature and $\sigma$ is equi-affine arc length parameter. It is worth mentioning that the area $A$ of the planar region that a curve encloses is variant under fully affine transformation, and one would feel uncertain whether there is an analogous formula in fully affine geometry.
    Inspired by the work of Gage \cite{gag} on the isoperimetric inequality with applications to curve shortening, and the work of Brendle \cite{bre} on the isoperimetric inequality for a minimal submanifold,
    we obtain, in the current paper, by employing the evolution process of equi-affine heat flow,
    a fully affine isoperimetric inequality (see Section \ref{sec-iso} for details)
    \begin{equation*}
       \oint d\xi\leq6\pi,
    \end{equation*}
    for any convex smooth embedded closed curve in $\R^2$, and equality holds only for the ellipse, where $\xi$ is fully affine arc length parameter
    and $\displaystyle\oint d\xi$ is the fully affine perimeter of the closed curve.
    In particular, the fully affine isoperimetric inequality plays a crucial role in the proof of Theorem \ref{thm-toelps}.

\subsection{Fully affine arc length in higher dimensions}
    Klein considered a geometry as being the study of a certain class for figures in a space of these properties which are left invariant with respect to same transitive group of transformations.
    The space is a homogeneous space $G/H$ with $G$ a Lie group, $H$ a closed subgroup, and with the group of transformations being the left action of $G$.
    The figures to be considered in this paper are smooth curves $[a,b]\stackrel{x}{\longrightarrow}G/H$.
    The natural objects of study are known classically as {\it differential invariants}, that is, local expressions in $x$ and its derivatives invariant under the action of $G$ (cf. \cite{gre}).
    They generalize arc-length, curvature and torsion for curves in Euclidean $\R^3$, which are differential invariants of orders $1,2$ and $3$ respectively.
    For the case of higher dimensional space $\R^n$, the arc length, curvatures (the invariants under reparametrization and transitive group of transformations) related to
    Euclidean group $\text{SE}(n)$, similarity group $\text{Sim}(n)$, centro-equi-affine group  $\text{SL}(n)$, centro-affine group $\text{GL}(n)$ and
    equi-affine group $\text{SA}(n)$ have been well developed (see Section \ref{invariant-group} for details).
    To the author's knowledge, the arc length and curvatures of curves in fully affine transformation group are only presented in two and three dimensional space \cite{ks}.
    In this paper, we extend the fully affine invariant theories of curves into the general dimensional space $\R^n$.
    It comes as a surprise to the author that, there is a sharp distinction in the expressions for fully affine arc length between different dimensional space (see the following examples and theorem for details),
    which is not as straightforward as these appearing in Section \ref{invariant-group}.
    Let us first observe the fully affine arc length element in $\R^2$ and $\R^3$ (the notation $\displaystyle x_{p^n}$ represents $\displaystyle \frac{d^nx}{dp^n}$ and $\displaystyle \bm{[}v_1,\cdots,v_n\bm{]}$ is denoted as a determinant form in $\R^n$, where the columns $v_1,\cdots,v_n$ are the  vectors in $\R^n$).
    \begin{exmp}\label{exm-f-r2}
    Consider the fully affine group $\text{GA}(2)=\text{GL}(2)\ltimes\R^2$ acting on $x:p\rightarrow x(p)\in\R^2$.
    Assume $x_p, \;x_{p^2}$ are linearly independent and $3\bm{[}x_p,x_{p^2}\bm{]}\bm{[}x_p,x_{p^4}\bm{]}-5\bm{[}x_p,x_{p^3}\bm{]}^2+12\bm{[}x_p,x_{p^2}\bm{]}\bm{[}x_{p^2},x_{p^3}\bm{]}\neq0$.
    Then the corresponding function
    \begin{equation*}
    f(p)=\sqrt{\epsilon\frac{3\bm{[}x_p,x_{p^2}\bm{]}\bm{[}x_p,x_{p^4}\bm{]}-5\bm{[}x_p,x_{p^3}\bm{]}^2+12\bm{[}x_p,x_{p^2}\bm{]}\bm{[}x_{p^2},x_{p^3}\bm{]}}{\bm{[}x_p,x_{p^2}\bm{]}^2}}
    \end{equation*}
    satisfies the conditions in Section \ref{sec-arc},
    and the fully affine arc length element for plane curves in $\R^2$ is given by
    \begin{equation*}
       ds=\sqrt{\epsilon\frac{3\bm{[}x_p,x_{p^2}\bm{]}\bm{[}x_p,x_{p^4}\bm{]}-5\bm{[}x_p,x_{p^3}\bm{]}^2+12\bm{[}x_p,x_{p^2}\bm{]}\bm{[}x_{p^2},x_{p^3}\bm{]}}{\bm{[}x_p,x_{p^2}\bm{]}^2}}dp,
    \end{equation*}
    where $\displaystyle \epsilon=\text{sgn}\left(3\bm{[}x_p,x_{p^2}\bm{]}\bm{[}x_p,x_{p^4}\bm{]}-5\bm{[}x_p,x_{p^3}\bm{]}^2+12\bm{[}x_p,x_{p^2}\bm{]}\bm{[}x_{p^2},x_{p^3}\bm{]}\right).$
    \end{exmp}
    \begin{exmp}\label{exm-f-r3}
     Consider the fully affine group $\text{GA}(3)=\text{GL}(3)\ltimes\R^3$ acting on $x:p\rightarrow x(p)\in\R^3$.
    Assume $x_p,x_{p^2}$ and $x_{p^3}$ are linearly independent, and
    $A-B+C\neq0,$
    where $A=24\bm{[}x_p,x_{p^2},x_{p^5}\bm{]}\bm{[}x_p,x_{p^2},x_{p^3}\bm{]}$, $B=35\bm{[}x_p,x_{p^2},x_{p^4}\bm{]}^2$ and $C=60\bm{[}x_p,x_{p^3},x_{p^4}\bm{]}\bm{[}x_p,x_{p^2},x_{p^3}\bm{]}$.
    One can verify that
     the fully affine arc length element in $\R^3$ can be defined as
    $ds=\sqrt{\epsilon\frac{A-B+C}{\bm{[}x_p,x_{p^2},x_{p^3}\bm{]}^2}}dp,$
    where
    $\displaystyle \epsilon=\text{sgn}\left(A-B+C\right).$
    \end{exmp}
     According to Example \ref{exm-f-r2} and Example \ref{exm-f-r3}, we will prove the following relation under the assumption that $x_p,x_{p^2},\cdots, x_{p^n}$ are linearly independent
     for the purpose of obtaining the fully affine arc length element of the curve $x:p\rightarrow x(p)\in\R^n$ (see Section \ref{affine-high} for details).
    \begin{thm}\label{thm-higher}
    Under the reparametrization $r=r(p)$, we have
    \begin{align*}
     &\alpha\frac{\bm{[}x_{p^{n+2}},x_{p^{n-1}},\cdots,x_{p^2},x_p\bm{]}}{\bm{[}x_{p^n},x_{p^{n-1}},\cdots,x_{p^2},x_p\bm{]}}
     -\beta\left(\frac{\bm{[}x_{p^{n+1}},x_{p^{n-1}},\cdots,x_{p^2},x_p\bm{]}}{\bm{[}x_{p^n},x_{p^{n-1}},\cdots,x_{p^2},x_p\bm{]}}\right)^2\\
     &\qquad\qquad\qquad\qquad\qquad\qquad\qquad\;+\gamma\frac{\bm{[}x_{p^{n+1}},x_{p^n},x_{p^{n-2}},\cdots,x_{p^2},x_p\bm{]}}{\bm{[}x_{p^n},x_{p^{n-1}},\cdots,x_{p^2},x_p\bm{]}}\\
     &=\left(\frac{dr}{dp}\right)^2
     \Bigg(\alpha\frac{\bm{[}x_{r^{n+2}},x_{r^{n-1}},\cdots,x_{r^2},x_r\bm{]}}{\bm{[}x_{r^n},x_{r^{n-1}},\cdots,x_{r^2},x_r\bm{]}}
     -\beta\left(\frac{\bm{[}x_{r^{n+1}},x_{r^{n-1}},\cdots,x_{r^2},x_r\bm{]}}{\bm{[}x_{r^n},x_{r^{n-1}},\cdots,x_{r^2},x_r\bm{]}}\right)^2\\
     &\qquad\qquad\qquad\qquad\qquad\qquad\qquad\;+\gamma\frac{\bm{[}x_{r^{n+1}},x_{r^n},x_{r^{n-2}},\cdots,x_{r^2},x_r\bm{]}}{\bm{[}x_{r^n},x_{r^{n-1}},\cdots,x_{r^2},x_r\bm{]}}\Bigg),
    \end{align*}
    where $\alpha=\frac{n(n+1)(n-1)}{\omega}$, $\beta=\frac{(n-1)(n+2)(2n+1)}{2\omega}$, $\gamma=\frac{n(n+1)(n+2)}{\omega}$, and $\omega$ is the greatest common divisor of
    $n(n+1)(n-1), \frac{(n-1)(n+2)(2n+1)}{2}$ and $n(n+1)(n+2)$.
    \end{thm}

\subsection{Heat flow in fully affine geometry}
    The term ``invariant submanifold flow'' is conceived of as the motion of one submanifold governed by a prescribed partial differential equation that admits an underlying transformation group as a symmetry group, e.g.~the Euclidean group of rigid motions (translations and rotations). There have been a number of achievements contributing to invariant geometric flows for curves and surfaces in Euclidean geometry and in affine geometry, specially about geometric heat flow. The curve shortening flow (CSF) is one of the simplest and well-studied models, which was introduced firstly by Mullins \cite{mul} as a model for the motion of grain boundaries. Later Gage and Hamilton \cite{gh} proved that a convex curve embedded in ${\mathbb R}^2$ shrinks to a point  when evolving under CSF.  In the sequel, the evolution of non-convex embedded curves was studied by Grayson \cite{gra-1,gra-2}, and it was proved that if the initial curve is any embedded curve in ${\mathbb R}^2$, then the corresponding curve first becomes convex and finally shrinks to a point in finite time while becoming asymptotically circular, often referred to as a ``circular point''. The higher dimensional analogue of CSF is the mean curvature flow (MCF), which was first investigated by Brakke \cite{bra}, in the context of geometric measure theory. Huisken \cite{hui-84} proved a smooth, compact and convex hypersurface without boundary converge to a round sphere after appropriate rescaling when evolving under MCF. The Ricci Flow (RF) was introduced by Hamilton \cite{ham-82}, which deforms an initial metric in the direction of its Ricci tensor. MCF shares many characteristics with CSF, and they are called geometric heat flows. Geometric heat flow means that we consider a family of submanifolds
    $F:M\times(0,T)\rightarrow N$ which solves the partial differential equation
    \begin{equation*}
       \frac{\partial}{\partial t}F=\Delta F,
    \end{equation*}
     with initial condition $F(\cdot, 0)=F_0:M\rightarrow N$, where $\Delta$ is the Laplace-Beltrami operator on $(M,g)$, $g$ denotes the metric on $M$.
     The Ricci flow is the analogue of the heat equation on a Riemannian manifold \cite{ds},  which was  extended by Perelman to in his famous solution to the Poincar\'e conjecture. In 2003, Perelman \cite{pre-1,pre-2,pre-3} completed Hamilton's Ricci flow programme \cite{ham} with the
     aim of settling Thurston's geometrization conjecture \cite{thu} for closed 3-manifolds.
     This conjecture had predicted such manifolds to be decomposable into pieces with
     locally homogeneous geometry.

       There are some analogical extensions of the geometric heat flow in affine geometry. The corresponding affine curve shortening flow (ACSF) was firstly introduced by  Sapiro, and Tannenbaum \cite{st}, and was further investigated by Angenent, Sapiro, and Tannenbaum \cite{ast}. It was shown  that any convex smooth embedded curve when evolves according to ACSF converges to an elliptical point \cite{ast,cz}.  Andrews \cite{and} studied an affine-geometric, fourth-order parabolic evolution equation
       for closed convex curves in the plane and proved the
        evolving curve remains strictly convex while expanding to infinite size and approaching a
        homothetically expanding ellipse.  More recently, similar results were obtained for the heat flow in  centro-equi-affine geometry \cite{wwq} and in centro-affine geometry \cite{oqy}. Interestingly,  the heat flow for the centro-affine curvature  is equivalent to the well-known inviscid Burgers' equation. In contrast with the heat flows in Euclidean, equi-affine, and centro-equi-affine geometries, those yield second order nonlinear  parabolic equations for the associated  invariant curvature. Heat flows in more general Klein geometries were investigated \cite{olv-isf, ost-1}.

      Some special solutions are vital for the detailed analysis of geometric flow. For instance, the grim-reapers were used by Grayson \cite{gra-3}, Altschuler \cite{alt} and Hamilton \cite{ham} for examining the behaviour of CSF.  The self-similar solutions and spiral wave solutions of CSF were served to study the formation of singularity of CSF \cite{ast, ham}. The Abresch-Langer curves were classified in \cite{al}, and the group-invariant solutions for CSF were classified in \cite{cl}.
     The notion of solitons of the flow, which means the surprising emergence of special evolutions which retain a fixed profile for all time, is intimately
     connected with the set of possible singularities. The classification of solitons to geometric flow is a central problem with important implications for the analysis of singularities.

     Thus another part of this paper is to study fully affine heat flow, which is a fully affine analogue of CSF and MCF in Euclidean geometry, and ACSF in equi-affine geometry. We find  the heat flow for the  fully affine curvature generates a fourth order nonlinear  parabolic equation. A complete classification of solitons for  this heat flow is provided. Moreover, we investigate the local existence, uniqueness, and long-term behavior of this fully affine heat flow.
       Chou in \cite{cho} used energy method to study the fourth order nonlinear  parabolic equation on Euclidean plane. By utilizing the similar technique and fully affine isoperimetric inequality we show that a closed embedded curve may converge to an ellipse when evolving according to the fully affine heat flow.
      The main results of this part are the followings:
     \begin{thm}\label{thm-slt}
      Under fully affine heat flow \eqref{heat-evo}, the expanding, translating and shrinking solitons are composed of
      \begin{itemize}
        \item[(1)] the curves with constant fully affine curvature;
        \item[(2)] the curves with the fully affine curvature $\displaystyle \varphi=-\frac{3}{\xi}$,
      $\displaystyle \varphi=A\tan\left(\frac{A}{3}\xi\right)$, $\displaystyle \varphi=-A\cot\left(\frac{A}{3}\xi\right)$, $\displaystyle \varphi=-A\tanh\left(\frac{A}{3}\xi\right)$ or $\displaystyle \varphi=-A\coth\left(\frac{A}{3}\xi\right)$,
      where $\xi$ is the fully affine arc length parameter and $A$ is an arbitrary nonzero constant.
      \end{itemize}
      In particular, the closed solitons for \eqref{heat-evo} are the ellipses.
    \end{thm}
     \begin{thm}\label{thm-toelps}
    Assume $C(\cdot,t)$ is a  solution of fully affine heat flow \eqref{heat-evo} in a maximal interval $[0,\omega),\;\omega\leq\infty$, where $C_0$ is a closed smooth embedded curve. Then the solution exists as long as the $L^2$-norm of the curvature $\varphi$ of $C(\cdot,t)$ is finite. Furthermore, when $\omega$ is finite,
    \begin{equation*}
      \oint_{C(\cdot,t)}\varphi^2d\xi\geq D(\omega-t)^{-1/4}
    \end{equation*}
    for some constant $D$. When $\omega$ is infinity,
    the curvature $\varphi$ of $C(\cdot,t)$ converges smoothly to zero, that is, $C(\cdot,t)$ converges to an ellipse.
    \end{thm}

\subsection{Organization of the paper}
     This paper is organized as follows. In Section 2, we recall the relevant definitions, notions and basic facts for curves related to
     Euclidean group $\text{SE}(n)$, similarity group $\text{Sim}(n)$, centro-equi-affine group  $\text{SL}(n)$, centro-affine group $\text{GL}(n)$ and equi-affine group $\text{SA}(n)$. In Section 3, we study fully affine differential invariants in $\R^n$ based on the Lie group $\text{GA}(n) = \text{GL}(n) \semidirect \Rn$.  In Section 4, under the motions of planar curve, the evolution formulas of fully affine differential invariants are derived. In Section 5,  we obtain the second variation formula of planar curves in fully affine space $\R^2$ and determine the stability of the extremal curves. In Section 6, the fully affine isoperimetric inequality is investigated.  In Section 7, we provide a complete classification of solitons for
     fully affine heat flow, and the local existence, uniqueness, and long-term behavior of this heat flow are discussed. In Appendix, we derive the motions of curves in equi-affine setting, and list a theorem in \cite{gi} of local existence for a fourth-order parabolic equation.

\subsection{Acknowledgements} Y. Yang was supported by the Fundamental Research Funds for the Central Universities under grant-N2104007, and he would also like to express his deep gratitude to Professor Peter J. Olver for his encouragement and help during
his stay in School of Mathematics, University of Minnesota as a Visiting Professor, while part of this work was completed.
This paper is dedicated to  Professor Peter J. Olver on the occasion of his 70th birthday.

\section{Arc length and differential invariants related to group}\label{invariant-group}

     Fully affine differential geometry is based on the Lie group $\text{GA}(n) = \text{GL}(n) \ltimes \R^n$ consisting of affine transformations $x \longmapsto Ax+b$, $A\in \text{GL}(n)$, $b\in \Rn$ acting on $x \in \Rn$.
     Before discussing the fully affine geometry, let us make a digression to review some notions and
     basic facts in its sub-geometries (Euclidean geometry, similarity geometry, centro-equi-affine geometry, centro-affine geometry and equi-affine geometry).
 \subsection{Group-invariant arc length element}\label{sec-arc}
     Assume throughout this section that all of our mappings
     are sufficiently smooth, so that all the relevant derivatives are well-defined.
    In the following, $\displaystyle \bm{[}v_1,\cdots,v_n\bm{]}$ is denoted as a determinant form in $\R^n$, where the columns $v_1,\cdots,v_n$ are the  vectors in $\R^n$.
    $v\cdot w$ or $\ip{v}{w}$ represents the inner product of vectors $v,w\in\R^n$,  $\displaystyle \|v\|:=\sqrt{\ip{v}{v}}$, and $\displaystyle x_{p^n}:=\frac{d^nx}{dp^n}$.

     Let $G$ be a Lie group acting on a curve $x:p\rightarrow x(p)\in\R^n$.
     Note that the order of a differential invariant or a function is the order of the highest
     derivative that occurs in the local expression for it.
     The arc length element with respect to $G$ can be determined by finding a differential $1$-form $f(p)dp$,
    whose invariance indicates that
    \begin{itemize}
      \item[(a)]$f(p)$ is invariant under the transformation $\bar{x}=g\circ x$, $\forall g\in G$,
      \item[(b)]with a parametrization $r=r(p)$, $f(p)dp=\bar{f}(r)dr$ (invariance of the tensor under a change of the basis $\displaystyle \frac{\partial}{\partial p}$).
    \end{itemize}
    Note that without loss of generality, we may assume $\displaystyle \frac{dr}{dp}>0$.
    Then the arc length element can be represented as
    \begin{equation}\label{grp-arc}
    ds:=f(p)dp,
    \end{equation}
    which guarantees the integral $\displaystyle \int^a_bf(p)dp$ is invariant of reparametrization and group of transformations.
    Here are some examples to clarify this point.
    \begin{exmp}\label{exm-se}
    As a most simple one illustrating our concern, we consider the Euclidean group $\text{SE}(n)=\text{SO}(n)\ltimes\R^n$ (a group consisting of the orientation preserving rigid motions of $\R^n$) acting on $x(p)$.
    Assume $\displaystyle\left\|\frac{dx}{dp}\right\|\neq 0$. Obviously, the function $\displaystyle f(p)=\left\|\frac{dx}{dp}\right\|$ satisfies the conditions (a) and (b), and it is well known that the Euclidean arc length element is denoted as
    $\displaystyle ds=\left\|\frac{dx}{dp}\right\|dp.$
    \end{exmp}
    \begin{exmp}\label{exm-sim}
    Any {\it similarity} $F:\R^n\rightarrow\R^n$ can be expressed in the form $F(x)=\lambda Ax+b$, where $x\in\R^n$ is an arbitrary point, $A$ is an orthogonal $n\times n$ matrix, $b$ is a translation vector, and the real constant $\lambda\neq0$. The group $\text{Sim}(n)$ consists of all orientation-preserving similarities of $\R^n$.
    The arc length element for group $\text{Sim}(n)$ acting on $x(p)$ has been studied in \cite{cq} and \cite{eg}. Suppose that $\displaystyle\left\|\frac{dx}{dp}\right\|\neq 0$ and $\displaystyle\left\|x_{p^2}-\frac{\ip{x_p}{x_{p^2}}}{\ip{x_p}{x_p}}x_p\right\|\neq0$. It is easy to verify the function
    $\displaystyle \frac{\left\|x_{p^2}-\frac{\ip{x_p}{x_{p^2}}}{\ip{x_p}{x_p}}x_p\right\|}{\|x_p\|}$ meets the conditions (a) and (b). Thus in similarity geometry, we have the arc length element
    $\displaystyle
     ds=\frac{\left\|x_{p^2}-\frac{\ip{x_p}{x_{p^2}}}{\ip{x_p}{x_p}}x_p\right\|}{\|x_p\|}dp,
    $
    which is coincident with the definition of similarity arc length in \cite{cq} and \cite{eg}.
    \end{exmp}
    \begin{exmp}\label{exm-cea}
     The differential geometry invariant to the action of special linear group $\text{SL}(n)$ is called {\it centro-equi-affine} differential geometry.
    Consider the group $\text{SL}(n)$ acting on $x(p)$, and suppose that  $x,x_p,x_{p^2},\cdots,x_{p^{n-1}}$ are linear independent.
    The corresponding function $f(p)$ occurring in \eqref{grp-arc} can be obtained by $\displaystyle \left(\bm{[}x,x_p,x_{p^2},\cdots,x_{p^{n-1}}\bm{]}\right)^{\frac{2}{n(n-1)}}$,  and it is easy to see that $\displaystyle
       ds=\left(\bm{[}x,x_p,x_{p^2},\cdots,x_{p^{n-1}}\bm{]}\right)^{\frac{2}{n(n-1)}}dp$ is defined as the centro-equi-affine arc length element (also see \cite{gs,olv-5}).
    \end{exmp}
    \begin{exmp}\label{exm-ca}
    \emph{Centro-affine} differential geometry refers to the geometry induced by the general linear group $x \longmapsto Ax$, $A\in \text{GL}(n)$, $x \in \Rn$, which is the subgroup of the affine transformation group that keeps the origin fixed. Consider the the general linear group $\text{GL}(n)$ acting on $x(p)$, and suppose $\bm{[}x,x_{p},\cdots,x_{p^{n-1}}\bm{]}\neq0$, $\bm{[}x_p,x_{p^2},\cdots,x_{p^n}\bm{]}\neq0$.
    The corresponding function $f(p)$ occurring in \eqref{grp-arc} can be expressed by $\displaystyle\left(\epsilon\frac{\bm{[}x_p,x_{p^2},\cdots,x_{p^n}\bm{]}}{\bm{[}x,x_{p},\cdots,x_{p^{n-1}}\bm{]}}\right)^{\frac{1}{n}}$, where
    $\displaystyle\epsilon=\text{sgn}\left(\frac{\bm{[}x_p,x_{p^2},\cdots,x_{p^n}\bm{]}}{\bm{[}x,x_{p},\cdots,x_{p^{n-1}}\bm{]}}\right)$, and the centro-affine arc length element (also see \cite{gw}) is
    $\displaystyle
       ds=\left(\epsilon\frac{\bm{[}x_p,x_{p^2},\cdots,x_{p^n}\bm{]}}{\bm{[}x,x_{p},\cdots,x_{p^{n-1}}\bm{]}}\right)^{\frac{1}{n}}dp.
    $
    \end{exmp}
    \begin{exmp}\label{exm-ea}
    {\it Equi-affine} differential geometry refers to volume-preserving affine differential geometry in which one restricts to the
    subgroup $\text{SA}(n)=\text{SL}(n)\ltimes\R^n$.
    Consider the group $\text{SA}(n)$ acting on $x(p)$, and let us assume $\displaystyle \bm{[}x_p,x_{p^2},\cdots,x_{p^n}\bm{]}\neq0$. By \cite{cot, dav0,wan1,wan2}, we find
    the corresponding function $f(p)$ occurring in \eqref{grp-arc} can be obtained though $\displaystyle \left(\bm{[}x_p,x_{p^2},\cdots,x_{p^n}\bm{]}\right)^{\frac{2}{n(n+1)}}$, and $\displaystyle
       ds=\left(\bm{[}x_p,x_{p^2},\cdots,x_{p^n}\bm{]}\right)^{\frac{2}{n(n+1)}}dp
    $  is considered as the equi-affine arc length element.
    \end{exmp}
 \subsection{The curvatures of a curve related to the group}
    All others differential invariants can be found by differentiation
    with respect to a group-invariant arc length element. In particular,
    once defining the group-invariant arc length parameter for a curve, the curvatures (or differential invariants) can be generated by the successive derivatives of the curve $x$ with respect to the arc length parameter and
    the linearly dependent coefficients. We still provide some examples to demonstrate its validity.
 \subsubsection{Euclidean differential invariants}
    Let $I\subset \R$ and consider a curve $x:I\rightarrow \R^n$ parametrized by Euclidean arc length $s$ defined as in Example \ref{exm-se}.
    The tangent vector field $V_1$ can be defined as the unit vector in the direction of $\displaystyle\frac{dx}{ds}$.
    The second basic unit vector field $V_2$ lies in the subspace spanned by the vector fields $\displaystyle\left\{\frac{dx}{ds}, \frac{d^2x}{ds^2}\right\}$,
    is perpendicular to $V_1$ and together with $V_1$ spans an area of $1$. The third basis vector, $V_3$, is in the subspace spanned by
    $\displaystyle\left\{\frac{dx}{ds}, \frac{d^2x}{ds^2},\frac{d^3x}{ds^3}\right\}$, is of unit length, is perpendicular to $V_1$ and $V_2$ and together with $V_1$ and $V_2$ spans a volume of $1$. Proceeding in this fashion, the $(k+1)$st basis vector is in the space spanned by $\displaystyle\left\{\frac{d^ix}{ds^i}: i=1,2,\cdots,k+1\right\}$, is of unit length, is perpendicular to $\{V_i:i=1,2,\cdots,k\}$, and together with $\{V_i:i=1,2,\cdots,k\}$ spans a volume of $1$.

    Given a smooth curve parametrized by Euclidean arc length $s$, the Euclidean curvature is given by
    $\displaystyle\kappa=\frac{dV_1}{ds}\cdot V_2$ and the higher Euclidean torsions are given by
    $\displaystyle\tau_i=\frac{dV_{i+1}}{ds}\cdot V_{i+2}, i=1,\cdots, n-2$. As shown in \cite{dav0}, the Frenet-Serret formulas are described by
    \begin{small}
    \begin{align*}
    \frac{dx}{ds}&=\;V_1,\\
    \frac{dV_1}{ds}&=\;\kappa V_2,\\
    \frac{dV_2}{ds}&=\;-\kappa V_2+\tau_1V_3,\\
    &\quad\cdots,  \cdots\\
    \frac{dV_{n-1}}{ds}&=\;-\tau_{n-3}V_{n-1}+\tau_{n-2}V_n,\\
    \frac{dV_{n}}{ds}&=\;-\tau_{n-2}V_{n-1}.
    \end{align*}
    \end{small}%
 \subsubsection{Similarity differential invariants}
    Assume the curve $x:I\rightarrow \R^n$ is parametrized by similarity arc length $s$ defined as in Example \ref{exm-sim}.
    We define the tangent $V_1$ as $\displaystyle\frac{dx}{ds}$.
    The second basic vector $V_2$ is in the subspace spanned by $\displaystyle\left\{\frac{dx}{ds}, \frac{d^2x}{ds^2}\right\}$, and
    is perpendicular to $V_1$. The third basis vector, $V_3$, is in the subspace spanned by
    $\displaystyle\left\{\frac{dx}{ds}, \frac{d^2x}{ds^2},\frac{d^3x}{ds^3}\right\}$, and is perpendicular to $V_1$ and $V_2$ . Proceeding in this fashion, the $(k+1)$st basis vector is in the space spanned by $\displaystyle\left\{\frac{d^ix}{ds^i}: i=1,2,\cdots,k+1\right\}$, and is perpendicular to $\{V_i:i=1,2,\cdots,k\}$.

    Hence, given a smooth curve parametrized by similarity arc length $s$, the similarity curvatures are generated by
    $\displaystyle\alpha_1=-\frac{\frac{dV_1}{ds}\cdot V_1}{V_1\cdot V_1}$ and
    $\displaystyle\alpha_i=\frac{\frac{dV_{i}}{ds}\cdot V_{i+1}}{V_{i+1}\cdot V_{i+1}}$, $i=2,\cdots, n-1$. In \cite{cq}, the following Frenet-Serret formulas for curves in similarity geometry have been established
    \begin{small}
    \begin{align*}
    \frac{dx}{ds}&=\;V_1,\\
    \frac{dV_1}{ds}&=\;-\alpha_1V_1+V_2,\\
    \frac{dV_2}{ds}&=\;-V_1-\alpha_1V_2+\alpha_2V_3,\\
    \frac{dV_3}{ds}&=\;-\alpha_2V_2-\alpha_1V_3+\alpha_3V_4,\\
    &\quad\cdots, \cdots\\
    \frac{dV_{n-1}}{ds}&=\;-\alpha_{n-2}V_{n-2}-\alpha_1V_{n-1}+\alpha_{n-1}V_n,\\
    \frac{dV_{n}}{ds}&=\;-\alpha_{n-1}V_{n-1}-\alpha_1V_n.
    \end{align*}
    \end{small}%
 \subsubsection{Centro-equiaffine differential invariants}
    For the curve $x:I\rightarrow\R^n$, by choosing the centro-equi-affine arc length parameter $s$ defined as in Example \ref{exm-cea}, we can see
    \begin{small}
    \begin{equation*}
      \bm{[}x,x_{s},\cdots,x_{s^{n-1}}\bm{]}=\epsilon, \qquad \epsilon=1\;\mathrm{or} \;-1.
    \end{equation*}
    \end{small}%
    Differentiating with respect to $s$  yields
    \begin{small}
    \begin{equation*}
      \bm{[}x,x_{s},\cdots,x_{s^{n-2}},x_{s^n}\bm{]}=0.
    \end{equation*}
    \end{small}%
    Hence it follows
    \begin{small}
    \begin{equation*}
      x_{s^n}=\mu_0x+\mu_1x_s+\mu_2x_{s^2}+\cdots+\mu_{n-2}x_{s^{n-2}}.
    \end{equation*}
    \end{small}%
    The functions
    \begin{small}
    \begin{equation*}
      \mu_0=\epsilon\bm{[}x_{s^n},x_s,\cdots,x_{s^{n-1}}\bm{]},\quad \mu_i=\epsilon\bm{[}x,x_s,\cdots,x_{s^{i-1}},x_{s^n},x_{s^{i+1}},\cdots,x_{s^{n-1}}\bm{]},
    \end{equation*}
    \end{small}%
    are called the centro-equi-affine curvatures of $x$ (see \cite{gs,olv-5} for details), where $i=1,2,\cdots,n-2$.

 \subsubsection{Centro-affine differential invariants}
    Let $x:I\rightarrow\R^n$ be a curve parametrized by centro-affine arc length $s$ defined as in Example \ref{exm-ca}. According to Example \ref{exm-ca}, it is easy to see
    \begin{equation*}
    \frac{\bm{[}x_s,x_{s^2},\cdots,x_{s^n}\bm{]}}{\bm{[}x,x_{s},\cdots,x_{s^{n-1}}\bm{]}}=\epsilon,\qquad \epsilon=1\;\text{or} -1.
    \end{equation*}
    A direct computation shows
    \begin{equation*}
    x_{s^n}+(-1)^{n-1}\epsilon x=\mu_1x_s+\mu_2x_{s^2}+\cdots+\mu_{n-1}x_{s^{n-1}}.
    \end{equation*}
    Then the centro-affine differential invariants are given by
    \begin{equation*}
      \mu_i=\frac{\bm{[}x,x_{s},\cdots,x_{s^{i-1}},x_{s^n},x_{s^{i+1}},\cdots,x_{s^{n-1}}\bm{]}}{\bm{[}x,x_{s},\cdots,x_{s^{n-1}}\bm{]}},\qquad i=1,2,\cdots,n-1,
    \end{equation*}
    which are called centro-affine curvatures (see \cite{gw} for details).

 \subsubsection{Equi-affine differential invariants}\label{ea-ivs}
    Let $x:I\rightarrow\R^n$ be parametrized by equi-affine arc length $s$ defined as in Example \ref{exm-ea}, so that
    \begin{equation*}
      \bm{[}x_s,x_{s^2},\cdots,x_{s^n}\bm{]}=\epsilon,\qquad \epsilon=1\;\text{or} -1.
    \end{equation*}
    Differentiating with respect to $s$ then gives
    \begin{equation*}
      \bm{[}x_s,x_{s^2},\cdots,x_{s^{n-1}},x_{s^{n+1}}\bm{]}=0.
    \end{equation*}
    Therefore, it follows
    \begin{equation*}
      x_{s^{n+1}}=\mu_1x_s+\mu_2x_{s^2}+\cdots+\mu_{n-1}x_{s^{n-1}}.
    \end{equation*}
    The functions
    \begin{equation*}
      \mu_i=\epsilon\bm{[}x_s,\cdots,x_{s^{i-1}},x_{s^{n+1}},x_{s^{i+1}},\cdots,x_{s^n}\bm{]},\qquad i=1,2\cdots,n-1
    \end{equation*}
    are called the equi-affine curvatures of $x$ (see \cite{cot, dav0,wan1,wan2} for details).

\section{The fully affine arc length in higher dimensional space}\label{affine-high}

    Now, we utilize the approach mentioned in Section \ref{invariant-group} to determine the arc length element and differential invariants in fully affine geometry.
    In fact, the fully affine arc length element and differential invariants for plane curves and space curves have been achieved in \cite{ks}.
    Example \ref{exm-f-r2} and Example \ref{exm-f-r3} describe the fully affine arc length element for plane curves and space curves expressed by the same fashion as in Section \ref{invariant-group}.
    However, it is not as straightforward as those appearing in Section \ref{invariant-group} to derive the fully affine arc length element in higher dimensional space $\R^n$ (for details see Example \ref{exm-f-r2}, Example \ref{exm-f-r3} and Theorem \ref{thm-higher}).

    Let us begin to deduce the expressions of the fully affine arc length parameter and curvatures for the curves
    in $\R^n$ under the fully affine group $\mathrm{GA}(n)=\mathrm{GL}(n)\ltimes\R^n$.

    Suppose that $x:p\rightarrow x(p)\in \R^n$ is a curve in $n$-dimensional affine space $\R^n$ with parameter $p$.  For another parameter $r=r(p)$, we have
    \begin{small}
    \begin{align}\label{rp-1-4}
       \begin{aligned}
       x_p&=\;x_r\frac{dr}{dp},\\
       x_{p^2}&=\;x_{r^2}\left(\frac{dr}{dp}\right)^2+x_r\frac{d^2r}{dp^2},\\
       x_{p^3}&=\;x_{r^3}\left(\frac{dr}{dp}\right)^3+3x_{r^2}\frac{dr}{dp}\frac{d^2r}{dp^2}+x_r\frac{d^3r}{dp^3},\\
       x_{p^4}&=\;x_{r^4}\left(\frac{dr}{dp}\right)^4+6x_{r^3}\left(\frac{dr}{dp}\right)^2\frac{d^2r}{dp^2}       \\
       &\qquad\qquad\qquad\qquad\quad +x_{r^2}\left(3\left(\frac{d^2r}{dp^2}\right)^2+4\frac{dr}{dp}\frac{d^3r}{dp^3}\right)+x_r\frac{d^4r}{dp^4},
       \end{aligned}
    \end{align}
    \end{small}%
    where  $x_{p^k}=\frac{d^kx}{dp^k}$ and $x_{r^k}=\frac{d^kx}{dr^k}$ for a positive integer $k$.

    Note that in the following we use the representation $\displaystyle\binom{n}{k}$ to denote the binomial coefficient $\displaystyle\frac{n!}{(n-k)!k!}$. Then, a direct computation shows
    \begin{lem}\label{lem-rp}
    For a positive integer $k>4$, we have the following iterations
    \begin{small}
    \begin{align*}
       x_{p^k}&=\;x_{r^k}\left(\frac{dr}{dp}\right)^k+x_{r^{k-1}}\binom{k}{2}\left(\frac{dr}{dp}\right)^{k-2}\frac{d^2r}{dp^2}\\
             &\qquad\qquad\qquad\quad          +x_{r^{k-2}}\left(A_k\left(\frac{dr}{dp}\right)^{k-4}\left(\frac{d^2r}{dp^2}\right)^2+B_k\left(\frac{dr}{dp}\right)^{k-3}\frac{d^3r}{dp^3}\right)\\
             &\qquad\qquad\qquad\quad\; \mod \Big(x_{r^{k-3}},\;\cdots,\;x_r\Big),
    \end{align*}
    \end{small}%
    where $A_4=3$, $B_4=4$, $\displaystyle A_k=A_{k-1}+(k-3)\binom{k-1}{2}$ and $\displaystyle B_k=B_{k-1}+\binom{k-1}{2}$.
    Moreover, we have
    \begin{equation*}
    A_k=\frac{k(k-1)(k-2)(k-3)}{8},\qquad B_k=\frac{k(k-1)(k-2)}{6}.
    \end{equation*}
    \end{lem}
    \begin{lem}\label{lem-solP}
    Take $A_k$ and $B_k$ as in the previous lemma. The solutions to the system of equations with respect to $\alpha, \beta$ and $\gamma$
    \begin{small}
    \begin{equation*}
       \left\{
        \begin{aligned}
        &\alpha \binom{n+2}{2}-2\beta \binom{n+1}{2}+\gamma \binom{n}{2}=0,\\
        &\alpha B_{n+2}-\gamma B_{n+1}=0,\\
        &\alpha A_{n+2}-\beta \binom{n+1}{2}^2+\gamma\left(\binom{n+1}{2}\binom{n}{2}-A_{n+1}\right)=0,
        \end{aligned}
        \right.
    \end{equation*}
    \end{small}
    are $\alpha=\omega n(n+1)(n-1)$, $\beta=\omega\frac{(n-1)(n+2)(2n+1)}{2}$, $\gamma=\omega n(n+1)(n+2)$, where $\omega$ is an arbitrary constant.
    \end{lem}
    In the following, for the curve $x:p\rightarrow x(p)\in \R^n$, assume $x_p, x_{p^2}, \cdots, x_{p^n}$ are linear independent.
    \begin{proof}[Proof of Theorem \ref{thm-higher}]
    According to \eqref{rp-1-4} and Lemma \ref{lem-rp}, one can see
    \begin{small}
    \begin{align*}
       x_{p^{n+1}}&=\;x_{r^{n+1}}\left(\frac{dr}{dp}\right)^{n+1}+x_{r^n}\binom{n+1}{2}\left(\frac{dr}{dp}\right)^{n-1}\frac{d^2r}{dp^2}\\
             &\qquad\qquad\qquad\quad          +x_{r^{n-1}}\left(A_{n+1}\left(\frac{dr}{dp}\right)^{n-3}\left(\frac{d^2r}{dp^2}\right)^2+B_{n+1}\left(\frac{dr}{dp}\right)^{n-2}\frac{d^3r}{dp^3}\right)\\
             &\qquad\qquad\qquad\quad \mod \Big(x_{r^{n-2}},\;\cdots,\;x_r\Big),
    \end{align*}
    \end{small}%
    and
    \begin{small}
    \begin{align*}
       x_{p^n}=x_{r^n}\left(\frac{dr}{dp}\right)^n+x_{r^{n-1}}\binom{n}{2}\left(\frac{dr}{dp}\right)^{n-2}\frac{d^2r}{dp^2} \mod (x_{r^{n-2}},\;\cdots,\;x_r).
    \end{align*}
    \end{small}%
    Thus, it is not hard to verify the following several relations.
    \begin{small}
    \begin{equation*}
    \bm{[}x_{p^n},x_{p^{n-1}},\cdots,x_{p^2},x_p\bm{]}=\left(\frac{dr}{dp}\right)^{\frac{n(n+1)}{2}}\bm{[}x_{r^n},x_{r^{n-1}},\cdots,x_{r^2},x_r\bm{]}.
    \end{equation*}
    \begin{align*}
    \bm{[}x_{p^{n+2}},x_{p^{n-1}},\cdots,x_p\bm{]}&=\;\left(\frac{dr}{dp}\right)^{\frac{n(n-1)}{2}}\bm{[}x_{p^{n+2}},x_{r^{n-1}},\cdots,x_{r^2},x_r\bm{]}\\
    &=\;\left(\frac{dr}{dp}\right)^{\frac{n(n+1)}{2}+2}\bm{[}x_{r^{n+2}},x_{r^{n-1}},\cdots,x_{r^2},x_r\bm{]}\\
    &\qquad+\binom{n+2}{2}\left(\frac{dr}{dp}\right)^{\frac{n(n+1)}{2}}\frac{d^2r}{dp^2}\bm{[}x_{r^{n+1}},x_{r^{n-1}},\cdots,x_{r^2},x_r\bm{]}\\
    &\qquad\;+\left(\frac{dr}{dp}\right)^{\frac{n(n+1)}{2}-2}\left(A_{n+2}\left(\frac{d^2r}{dp^2}\right)^2+B_{n+2}\frac{dr}{dp}\frac{d^3r}{dp^3}\right)\\
    &\qquad\;\qquad\qquad\qquad\qquad\times\bm{[}x_{r^n},x_{r^{n-1}},\cdots,x_{r^2},x_r\bm{]}.
    \end{align*}
    \begin{align*}
    \bm{[}x_{p^{n+1}},x_{p^{n-1}},\cdots,x_p\bm{]}&=\;\left(\frac{dr}{dp}\right)^{\frac{n(n-1)}{2}}\bm{[}x_{p^{n+1}},x_{r^{n-1}},\cdots,x_{r^2},x_r\bm{]}\\
    &=\;\left(\frac{dr}{dp}\right)^{\frac{n(n+1)}{2}+1}\bm{[}x_{r^{n+1}},x_{r^{n-1}},\cdots,x_{r^2},x_r\bm{]}\\
    &\qquad\;+\left(\frac{dr}{dp}\right)^{\frac{n(n+1)}{2}-1}\frac{d^2r}{dp^2}\binom{n+1}{2}\bm{[}x_{r^n},x_{r^{n-1}},\cdots,x_{r^2},x_r\bm{]}.
    \end{align*}
    \begin{align*}
    \bm{[}x_{p^{n+1}},x_{p^n},x_{p^{n-2}},\cdots,x_p\bm{]}&=\;\left(\frac{dr}{dp}\right)^{\frac{(n-2)(n-1)}{2}}\bm{[}x_{p^{n+1}},x_{p^n},x_{r^{n-2}},\cdots,x_{r^2},x_r\bm{]}\\
    &=\;\left(\frac{dr}{dp}\right)^{\frac{n(n+1)}{2}+2}\bm{[}x_{r^{n+1}},x_{r^n},x_{r^{n-2}},\cdots,x_{r^2},x_r\bm{]}\\
    &+\binom{n}{2}\frac{d^2r}{dp^2}\left(\frac{dr}{dp}\right)^{\frac{n(n+1)}{2}}\bm{[}x_{r^{n+1}},x_{r^{n-1}},\cdots,x_{r^2},x_r\bm{]}\\
    &+\Bigg(\left(\frac{d^2r}{dp^2}\right)^2\left(\binom{n+1}{2}\binom{n}{2}-A_{n+1}\right)-\frac{dr}{dp}\frac{d^3r}{dp^3}B_{n+1}\Bigg)\\
    &\qquad\quad\times\left(\frac{dr}{dp}\right)^{\frac{n(n+1)}{2}-2}\bm{[}x_{r^n},x_{r^{n-1}},x_{r^{n-2}},\cdots,x_r\bm{]}.
    \end{align*}
    \end{small}%
    A trivial computation shows
    \begin{small}
    \begin{align*}
     &\alpha\frac{\bm{[}x_{p^{n+2}},x_{p^{n-1}},\cdots,x_{p^2},x_p\bm{]}}{\bm{[}x_{p^n},x_{p^{n-1}},\cdots,x_{p^2},x_p\bm{]}}
     -\beta\left(\frac{\bm{[}x_{p^{n+1}},x_{p^{n-1}},\cdots,x_{p^2},x_p\bm{]}}{\bm{[}x_{p^n},x_{p^{n-1}},\cdots,x_{p^2},x_p\bm{]}}\right)^2\\
     &\qquad\qquad\qquad\qquad\qquad\;+\gamma\frac{\bm{[}x_{p^{n+1}},x_{p^n},x_{p^{n-2}},\cdots,x_{p^2},x_p\bm{]}}{\bm{[}x_{p^n},x_{p^{n-1}},\cdots,x_{p^2},x_p\bm{]}}\\
     &=\left(\frac{dr}{dp}\right)^2
     \Bigg(\alpha\frac{\bm{[}x_{r^{n+2}},x_{r^{n-1}},\cdots,x_{r^2},x_r\bm{]}}{\bm{[}x_{r^n},x_{r^{n-1}},\cdots,x_{r^2},x_r\bm{]}}
     -\beta\left(\frac{\bm{[}x_{r^{n+1}},x_{r^{n-1}},\cdots,x_{r^2},x_r\bm{]}}{\bm{[}x_{r^n},x_{r^{n-1}},\cdots,x_{r^2},x_r\bm{]}}\right)^2\\
     &\qquad\qquad\qquad\qquad\qquad\;+\gamma\frac{\bm{[}x_{r^{n+1}},x_{r^n},x_{r^{n-2}},\cdots,x_{r^2},x_r\bm{]}}{\bm{[}x_{r^n},x_{r^{n-1}},\cdots,x_{r^2},x_r\bm{]}}\Bigg)+P_1+P_2,
    \end{align*}
    \end{small}%
    where
    \begin{small}
    \begin{align*}
     P_1&=\;\left(\alpha \binom{n+2}{2}-2\beta \binom{n+1}{2}+\gamma \binom{n}{2}\right)\frac{d^2r}{dp^2}\frac{\bm{[}x_{r^{n+1}},x_{r^{n-1}},\cdots,x_{r^2},x_r\bm{]}}{\bm{[}x_{r^n},x_{r^{n-1}},\cdots,x_{r^2},x_r\bm{]}},\\
     P_2&=\;\left(\alpha A_{n+2}-\beta \binom{n+1}{2}^2+\gamma\left(\binom{n+1}{2}\binom{n}{2}-A_{n+1}\right)\right)\left(\frac{dr}{dp}\right)^{-2}\left(\frac{d^2r}{dp^2}\right)^2\\
     &\qquad\;\qquad\qquad\qquad\qquad\qquad\qquad\qquad\qquad+\left(\alpha B_{n+2}-\gamma B_{n+1}\right)\left(\frac{dr}{dp}\right)^{-1}\frac{d^3r}{dp^3}.
    \end{align*}
    \end{small}
    Lemma \ref{lem-solP} implies $P_1=0$ and $P_2=0$, and then we complete the proof of Theorem \ref{thm-higher}.
    \end{proof}
    From now on, we denote
    \begin{small}
    \begin{align*}
       F:&=\;\alpha\frac{\bm{[}x_{p^{n+2}},x_{p^{n-1}},\cdots,x_{p^2},x_p\bm{]}}{\bm{[}x_{p^n},x_{p^{n-1}},\cdots,x_{p^2},x_p\bm{]}}
     -\beta\left(\frac{\bm{[}x_{p^{n+1}},x_{p^{n-1}},\cdots,x_{p^2},x_p\bm{]}}{\bm{[}x_{p^n},x_{p^{n-1}},\cdots,x_{p^2},x_p\bm{]}}\right)^2\\
     &\qquad\qquad\qquad\qquad\qquad\qquad\qquad\qquad\;+\gamma\frac{\bm{[}x_{p^{n+1}},x_{p^n},x_{p^{n-2}},\cdots,x_{p^2},x_p\bm{]}}{\bm{[}x_{p^n},x_{p^{n-1}},\cdots,x_{p^2},x_p\bm{]}},
    \end{align*}
    \end{small}%
    and assume
    $F\neq 0$
    for the curve $x(p)\in\R^n$,
    where
    \begin{equation*}
    \alpha=\frac{n(n+1)(n-1)}{\omega},\; \beta=\frac{(n-1)(n+2)(2n+1)}{2\omega}, \;\gamma=\frac{n(n+1)(n+2)}{\omega},
    \end{equation*}
     and $\omega$ is the greatest common divisor of
    $\displaystyle n(n+1)(n-1)$, $n(n+1)(n+2)$ and $\displaystyle\frac{(n-1)(n+2)(2n+1)}{2}$.
    Then we can define the fully affine arc length element in $\R^n$.
    \begin{defn}\label{defn-fal}
    The fully affine arc length element for the curve $x:p\rightarrow x(p)\in \R^n$ is defined as
    \begin{equation*}
    d\xi =\sqrt{\epsilon F}dp, \qquad \epsilon=\mathrm{sgn}(F).
    \end{equation*}
    \end{defn}
    \begin{rem}
     This definition for the fully affine arc length parameter in $\R^n$ is in accordance with the definitions in Example \ref{exm-f-r2} and Example \ref{exm-f-r3}.
    \end{rem}
    With the fully affine arc length parameter $\xi$, we have
    \begin{align}\label{arc-p}
       &\alpha\frac{\bm{[}x_{\xi^{n+2}},x_{\xi^{n-1}},\cdots,x_{\xi^2},x_\xi\bm{]}}{\bm{[}x_{\xi^n},x_{\xi^{n-1}},\cdots,x_{\xi^2},x_\xi\bm{]}}
     -\beta\left(\frac{\bm{[}x_{\xi^{n+1}},x_{\xi^{n-1}},\cdots,x_{\xi^2},x_\xi\bm{]}}{\bm{[}x_{\xi^n},x_{\xi^{n-1}},\cdots,x_{\xi^2},x_\xi\bm{]}}\right)^2\nonumber\\
     &\qquad\;\qquad\qquad\qquad\qquad+\gamma\frac{\bm{[}x_{\xi^{n+1}},x_{\xi^n},x_{\xi^{n-2}},\cdots,x_{\xi^2},x_\xi\bm{]}}{\bm{[}x_{\xi^n},x_{\xi^{n-1}},\cdots,x_{\xi^2},x_\xi\bm{]}}=\epsilon.
    \end{align}
    Assume
    \begin{align}\label{fre-1}
      x_{\xi^{n+1}}=\varphi_1x_{\xi^n}+\lambda x_{\xi^{n-1}}+\varphi_3x_{\xi^{n-2}}+\cdots+\varphi_nx_{\xi}.
    \end{align}
    It follows from \eqref{arc-p} that
    \begin{align*}
      (\alpha-\beta)\varphi_1^2+\alpha \frac{d\varphi_1}{d\xi}+(\alpha-\gamma)\lambda=\epsilon,
    \end{align*}
    that is,
    \begin{align}\label{fre-2}
     \lambda=\omega\frac{(\alpha-\beta)\varphi_1^2+\alpha \frac{d\varphi_1}{d\xi}-\epsilon}{3n(n+1)}.
    \end{align}
    Furthermore, by \eqref{fre-1}, we have
    \begin{align*}
      \varphi_1&=\;\frac{\bm{[}x_{\xi^{n+1}},x_{\xi^{n-1}},\cdots,x_{\xi}\bm{]}}{\bm{[}x_{\xi^n},x_{\xi^{n-1}},\cdots,x_{\xi}\bm{]}}, \\
      \varphi_3&=\;\frac{\bm{[}x_{\xi^n},x_{\xi^{n-1}},x_{\xi^{n+1}},x_{\xi^{n-3}},\cdots,x_{\xi}\bm{]}}{\bm{[}x_{\xi^n},x_{\xi^{n-1}},\cdots,x_{\xi}\bm{]}},\\
      &\cdots,\cdots\\
      \varphi_n&=\;\frac{\bm{[}x_{\xi^n},x_{\xi^{n-1}},\cdots,x_{\xi^2},x_{\xi^{n+1}}\bm{]}}{\bm{[}x_{\xi^n},x_{\xi^{n-1}},\cdots,x_{\xi}\bm{]}}.
    \end{align*}
    \begin{defn}\label{defn-cur}
      $\displaystyle\varphi_1,\varphi_3,\varphi_4,\cdots, \varphi_n$ in \eqref{fre-1} are called the fully affine curvatures of the curve in $\R^n$.
    \end{defn}
    By the existence and unique of the ordinary differential equation system and using the similar proof to that of Theorem 2.5 in \cite{ks}, one can deduce
    \begin{thm}
      Given functions $\displaystyle\varphi_i(\xi)$, $i=1,3,4,\cdots, n$ of a parameter $\xi$ and $\epsilon=\pm1$ with the relations \eqref{fre-1} and \eqref{fre-2}, there exists a curve $x(\xi)$ in $\R^n$ for which $\xi$ is a fully affine length parameter, and $\displaystyle\varphi_i(\xi)$, $i=1,3,4,\cdots, n$
      are the fully curvature functions uniquely up to a fully affine transformation.
    \end{thm}

\section{Evolutions of the fully affine differential invariants}\label{sec-Evo}

    For a curve $x:p\rightarrow x(p)\in\R^2$, if $\bm{[}x_p, x_{p^2}\bm{]}\neq0$ on the whole curve, we call this curve $x$ is {\it non-degenerate}.
    A point where $F=0$ is called a {\it fully affine inflection point}. Note that in this paper, we assume all curves are non-degenerate without fully affine inflection point.

    Consider a family of smooth curves  parameterized by $C:I_1\times I_2\to\R^2$, where $t\in I_2\subset\R$ can be viewed as the time parameter and $p\in I_1\subset\R$ is a free parameter of each individual curve in the family. Then
    by Definition \ref{defn-fal}, we have
    \begin{small}
    \begin{align}\label{affine-arc}
    g(p,t):=\sqrt{\epsilon\frac{3\bm{[}C_p,~C_{p^2}\bm{]}\bm{[}C_p,~C_{p^4}\bm{]}-5\bm{[}C_p,~C_{p^3}\bm{]}^2+12\bm{[}C_p,~C_{p^2}\bm{]}\bm{[}C_{p^2},~C_{p^3}\bm{]}}{\bm{[}C_p,~C_{p^2}\bm{]}^2}},
    \end{align}
    \end{small}%
    and the fully affine arc length along the curve is given by
    \begin{equation*}
    \xi=\int^p_0g(p,t) dp,
    \end{equation*}
    and we may use either $\{p,t\}$ or $\{\xi,t\}$ as coordinates of a point on the curve.
    According to Definition \ref{defn-cur}, the fully affine differential invariants $\varphi$ and $\lambda$ are given by
    \begin{equation}\label{affine-curv}
            \varphi(\xi,t)=-\frac{\bm{[}C_{\xi},~C_{\xi^3}\bm{]}}{\bm{[}C_\xi,~C_{\xi^2}\bm{]}} ,\qquad \lambda(\xi,t)=-\frac{\bm{[}C_{\xi^3},~C_{\xi^2}\bm{]}}{\bm{[}C_\xi,~C_{\xi^2}\bm{]}},
    \end{equation}
    where
    \begin{equation}\label{lam_exp}
    \lambda=\frac{2\varphi^2+3\varphi_{\xi}+\epsilon}{9},
    \end{equation}
    and $\varphi$ is called  fully affine curvature in $\R^2$.
    According to \eqref{fre-1}, we have the following equations,
    \begin{small}
    \begin{align}\label{c3xi}
    \begin{split}
    C_{\xi^3}&=-\lambda C_\xi-\varphi C_{\xi^2},\\
    C_{\xi^4}&=(\varphi\lambda - \lambda_\xi)C_\xi+(\varphi^2 - \lambda -\varphi_\xi)C_{\xi^2}.
    \end{split}
    \end{align}
    \end{small}%
    Note that in the subsequent calculations, \eqref{c3xi} shall be repeatedly applied to remove the higher derivatives $C_{\xi^i},\; i\geq 3$.
    Assume that the curve $C(p,t)$ evolves according to the curve flow
    \begin{small}
    \begin{equation}\label{curve-motion}
    \frac{\partial C}{\partial t}=WC_{\xi}+UC_{\xi^2},
    \end{equation}
    \end{small}%
    and the motion is said to be fully affine invariant if $\{W,U\}$ depend only on local values of $\{\varphi\}$ and its $\xi$ derivatives, that is , $W$ and $U$ are fully affine invariants.

 \subsection{The evolutions of arc length and curvature}
    The evolution in time of the other variables is determined by requiring
    \begin{align}\label{req-eq}
    \begin{split}
    \frac{\partial}{\partial t}\frac{\partial}{\partial p}&=\frac{\partial}{\partial p}\frac{\partial}{\partial t},\\
    \frac{\partial}{\partial p}&=g\frac{\partial}{\partial \xi},\\
    \frac{\partial}{\partial t}\frac{\partial }{\partial \xi}
                    &=\frac{\partial}{\partial t}\left(\frac{1}{g}\frac{\partial}{\partial p}\right)
                    =-\frac{g_t}{g}\frac{\partial}{\partial \xi}+\frac{\partial }{\partial \xi}\frac{\partial}{\partial t}.
    \end{split}
    \end{align}
    Let us compute the fully affine metric evolution. According to \eqref{affine-arc}, we form the expression
        \begin{equation}\label{gt_EFG}
          \epsilon\frac{\partial g^2}{\partial t}=3E-10F+12G,
        \end{equation}
    where
    \begin{small}
    \begin{align}\label{eq-EFG}
    \begin{split}
      &E=\frac{\Big(\bm{[}C_{pt},~C_{p^4}\bm{]}+\bm{[}C_p,~C_{p^4t}\bm{]}\Big)\bm{[}C_p,~C_{p^2}\bm{]}-\bm{[}C_p,~C_{p^4}\bm{]}\Big(\bm{[}C_{pt},~C_{p^2}\bm{]}+\bm{[}C_p,~C_{p^2t}\bm{]}\Big)}{\bm{[}C_p,~C_{p^2}\bm{]}^2},\\
      &F=\frac{\bm{[}C_p,~C_{p^3}\bm{]}}{\bm{[}C_p,~C_{p^2}\bm{]}}\times\\
              &\;\qquad\frac{\Big(\bm{[}C_{pt},~C_{p^3}\bm{]}+\bm{[}C_p,~C_{p^3t}\bm{]}\Big)\bm{[}C_p,~C_{p^2}\bm{]}-\bm{[}C_p,~C_{p^3}\bm{]}\Big(\bm{[}C_{pt},~C_{p^2}\bm{]}+\bm{[}C_p,~C_{p^2t}\bm{]}\Big)}{\bm{[}C_p,~C_{p^2}\bm{]}^2},\\
      &G=\frac{\Big(\bm{[}C_{p^2t},~C_{p^3}\bm{]}+\bm{[}C_{p^2},~C_{p^3t}\bm{]}\Big)\bm{[}C_p,~C_{p^2}\bm{]}-\bm{[}C_{p^2},~C_{p^3}\bm{]}\Big(\bm{[}C_{pt},~C_{p^2}\bm{]}+\bm{[}C_p,~C_{p^2t}\bm{]}\Big)}{\bm{[}C_p,~C_{p^2}\bm{]}^2}.
    \end{split}
    \end{align}
    \end{small}%
    In fact, by \eqref{c3xi} and \eqref{req-eq}, one easily verifies that
    \begin{small}
    \begin{align*}
    C_p&=\;gC_\xi,\quad
    C_{p^2}=\;gg_\xi C_\xi+g^2C_{\xi^2},\\
    C_{p^3}&=\;\Big(- g^3\lambda+gg^2_\xi  + g^2g_{\xi^2}\Big)C_\xi+\Big(-\varphi g^3+3g^2g_\xi\Big)C_{\xi^2},\\
    C_{p^4}&=\;\Big(g^4(\varphi\lambda-\lambda_\xi)+g^3(g_{\xi^3}- 6\lambda g_\xi)+4g^2g_\xi g_{\xi^2}+gg^3_\xi\Big)C_\xi\nonumber\\
                &\qquad       \qquad\qquad+\Big(g^4(\varphi^2-\lambda-\varphi_\xi)+g^3(4g_{\xi^2} - 6\varphi g_\xi)+7g^2g^2_\xi\Big)C_{\xi^2}.
    \end{align*}
    \end{small}%
    Recalling \eqref{req-eq}, we find, in view of \eqref{c3xi} and \eqref{curve-motion}
    \begin{small}
    \begin{align*}
    C_{pt}=\;&(WC_{\xi}+UC_{\xi^2})_p\nonumber\\
                   =\;&g(W_{\xi}C_{\xi}+WC_{\xi^2}+U_{\xi}C_{\xi^2}+UC_{\xi^3})\nonumber\\
                   =\;&g\Big(W_{\xi}-U\lambda\Big)C_{\xi}+g\Big(W+U_{\xi}-U\varphi\Big)C_{\xi^2},
    \end{align*}
    \begin{align*}
    C_{p^2t}=\;&\Big(g^2(W_{\xi^2} - U\lambda_{\xi} - 2\lambda U_{\xi} - W\lambda+ U\varphi\lambda) + gg_{\xi}(W_{\xi}  - U\lambda)\Big)C_{\xi}\nonumber\\
                     &\quad+\Big(g^2\left(2W_{\xi} + U_{\xi^2} - 2\varphi U_{\xi}+ U(\varphi^2- \varphi_{\xi}  - \lambda) - W\varphi\right) \nonumber\\
                     &\qquad\qquad\qquad\qquad\qquad\qquad\qquad+ gg_{\xi}(U_{\xi}  + W - U\varphi)\Big)C_{\xi^2},
    \end{align*}
    \begin{align*}
    C_{p^3t}=\;&\bigg(g^3\Big(W_{\xi^3} - 3\lambda W_{\xi}+ W(\varphi\lambda- \lambda_{\xi})  + U(\lambda^2- \varphi^2\lambda + \varphi\lambda_{\xi}+ 2\lambda\varphi_{\xi}- \lambda_{\xi^2})    \nonumber\\
                       & \qquad\quad + 3 U_{\xi}(\varphi\lambda- \lambda_{\xi})- 3\lambda U_{\xi^2} \Big)\nonumber\\
                       &\qquad+ g^2\Big(W_{\xi}g_{\xi^2} + 3g_{\xi}W_{\xi^2} + (3\varphi\lambda g_{\xi}- \lambda g_{\xi^2} - 3g_{\xi}\lambda_{\xi})U- 3(2 U_{\xi}+ W)\lambda g_{\xi}  \Big)\nonumber\\
                       &\qquad + gg_{\xi}^2\Big(W_{\xi} - U\lambda \Big) \bigg)C_{\xi}\nonumber\\
                       &\quad+\bigg(g^3\Big(U_{\xi^3} - 3\varphi U_{\xi^2}+ 3W_{\xi^2}    - 3\varphi W_{\xi} + (3U_{\xi}+W)(\varphi^2- \lambda - \varphi_{\xi})\nonumber\\
                       & \qquad\qquad\quad   + U(2\varphi\lambda + 3\varphi \varphi_{\xi}- 2\lambda_{\xi}- \varphi^3- \varphi_{\xi^2})\Big)  \nonumber\\
                       &\qquad\quad + g^2\Big(3g_{\xi}U_{\xi^2} +6W_{\xi}g_{\xi} + U_{\xi}(g_{\xi^2}- 6\varphi g_{\xi}) \nonumber\\
                       & \qquad\qquad\quad\quad+ U(3\varphi^2g_{\xi} - \varphi g_{\xi^2} - 3g_{\xi}\varphi_{\xi}- 3\lambda g_{\xi})+  W(g_{\xi^2} - 3\varphi g_{\xi})\Big)  \nonumber\\
                       & \qquad\quad  + gg_{\xi}^2\Big(W +  U_{\xi} - U\varphi \Big)\bigg)C_{\xi^2},
    \end{align*}
    \begin{align*}
    C_{p^4t}&=\bigg(
                         g^4\Big(W_{\xi^4}- 4\lambda U_{\xi^3} - 6\lambda W_{\xi^2}+ 6U_{\xi^2}(\varphi\lambda - \lambda_{\xi}) \nonumber\\
                          &\qquad\qquad+ 4U_{\xi}(\lambda^2- \varphi^2\lambda + \varphi \lambda_{\xi} + 2\lambda \varphi_{\xi}-\lambda_{\xi^2}) \nonumber\\
                         &\qquad\qquad+U(\varphi^3\lambda- \varphi^2\lambda_{\xi}- 2\varphi\lambda^2- 5\varphi\lambda\varphi_{\xi}\nonumber\\
                         &\qquad\qquad\qquad\qquad+ 3\varphi_{\xi}\lambda_{\xi} + 4\lambda\lambda_{\xi}+ \varphi\lambda_{\xi^2} + 3\lambda\varphi_{\xi^2}  - \lambda_{\xi^3} )\nonumber\\
                         &\qquad\qquad + 4W_{\xi}(\varphi\lambda - \lambda_{\xi})+ W(\lambda^2- \varphi^2\lambda + \varphi\lambda_{\xi} + 2\lambda\varphi_{\xi} - \lambda_{\xi^2}) \Big) \nonumber\\
                        &\qquad+g^3\Big(6g_{\xi}W_{\xi^3}-18g_{\xi}\lambda U_{\xi^2}+4g_{\xi^2}W_{\xi^2}+2U_{\xi}(9g_{\xi}\lambda\varphi -9g_{\xi}\lambda_{\xi} -4g_{\xi^2}\lambda)\nonumber\\
                        &\qquad\qquad +U(6g_{\xi}\lambda^2-6g_{\xi}\lambda\varphi^2+12g_{\xi}\lambda\varphi_{\xi}\nonumber\\
                        &\qquad\qquad\qquad\qquad+6g_{\xi}\lambda_{\xi}\varphi+4g_{\xi^2}\lambda\varphi-g_{\xi^3}\lambda
                        -6g_{\xi}\lambda_{\xi^2}-4g_{\xi^2}\lambda_{\xi})\nonumber\\
                        &\qquad\qquad+W_{\xi}(g_{\xi^3}-18g_{\xi}\lambda)+W(6g_{\xi}\lambda\varphi-6g_{\xi}\lambda_{\xi}-4g_{\xi^2}\lambda)
                        \Big)\nonumber\\
                        &\qquad+g^2\Big(7g_{\xi}^2W_{\xi^2}+4g_{\xi}g_{\xi^2}W_{\xi}-7g_{\xi}^2\lambda W\nonumber\\
                        &\qquad\qquad-14g_{\xi}^2\lambda U_{\xi}+U(7g_{\xi}^2\lambda\varphi-7g_{\xi}^2\lambda_{\xi}-4g_{\xi}g_{\xi^2}\lambda)
                        \Big)\nonumber\\
                        &\qquad+gg_{\xi}^3\Big(W_{\xi}-\lambda U\Big)
                        \bigg)C_{\xi}\nonumber\\
                        &\quad+\bigg(
                        g^4\Big(U_{\xi^4}+4W_{\xi^3}-4\varphi U_{\xi^3}-6\varphi W_{\xi^2}+6U_{\xi^2}(\varphi^2-\lambda-\varphi_{\xi})\nonumber\\
                        &\qquad\qquad+4W_{\xi}(\varphi^2-\lambda-\varphi_{\xi})+W(2\lambda\varphi-\varphi^3+3\varphi\varphi_{\xi}-2\lambda_{\xi}-\varphi_{\xi^2})\nonumber\\
                        &\qquad\qquad+4U_{\xi}(2\lambda\varphi-\varphi^3+3\varphi\varphi_{\xi}-2\lambda_{\xi}-\varphi_{\xi^2})\nonumber\\
                        &\qquad\qquad+U(\varphi^4-3\lambda\varphi^2-6\varphi^2\varphi_{\xi}+\lambda^2\nonumber\\
                        &\qquad\qquad\qquad\qquad+4\lambda\varphi_{\xi}+5\lambda_{\xi}\varphi+4\varphi\varphi_{\xi^2}+3\varphi_{\xi}^2-3\lambda_{\xi^2}-\varphi_{\xi^3})\Big)\nonumber\\
                        &\qquad+g^3\Big(6g_{\xi}U_{\xi^3}+18g_{\xi}W_{\xi^2}+2(2g_{\xi^2}-9g_{\xi}\varphi)U_{\xi^2}+2W_{\xi}(4g_{\xi^2}-9g_{\xi}\varphi)\nonumber\\
                        &\qquad\qquad+U_{\xi}(18g_{\xi}\varphi^2-8g_{\xi^2}\varphi-18g_{\xi}\lambda-18g_{\xi}\varphi_{\xi}+g_{\xi^3})\nonumber\\
                        &\qquad\qquad+W(6g_{\xi}(\varphi^2-\lambda-\varphi_{\xi})-4g_{\xi^2}\varphi+g_{\xi^3})+U\big(4g_{\xi^2}(\varphi^2-\lambda-\varphi_{\xi})\nonumber\\
                        &\qquad\qquad+6g_{\xi}(2\lambda\varphi-\varphi^3+3\varphi_{\xi}\varphi-2\lambda_{\xi}-\varphi_{\xi^2})-g_{\xi^3}\varphi\big)\Big)\nonumber\\
                        &\qquad+g^2\Big(7g_{\xi}^2U_{\xi^2}+14g_{\xi}^2W_{\xi}+W(4g_{\xi}g_{\xi^2}-7g_{\xi}^2\varphi)+2U_{\xi}(2g_{\xi}g_{\xi^2}-7g_{\xi}^2\varphi)\nonumber\\
                        &\qquad\qquad+U(7g_{\xi}^2\varphi^2-7g_{\xi}^2\varphi_{\xi}-7g_{\xi}^2\lambda-4g_{\xi}g_{\xi^2}\varphi)\Big)\nonumber\\
                        &\qquad+gg_{\xi}^3\Big(W+U_{\xi}-U\varphi\Big)\bigg)C_{\xi^2}.
    \end{align*}
    \end{small}%
    Then from above equations, a direct computation yields
    \begin{small}
    \begin{align*}
    \frac{\bm{[}C_{pt}, C_{p^4}\bm{]}}{\bm{[}C_{\xi},C_{\xi^2}\bm{]}}=\;&
    g^5\Big(U(\lambda^2+\varphi_{\xi}\lambda-\lambda_{\xi}\varphi)+U_{\xi}(\lambda_{\xi}-\varphi\lambda)\nonumber\\
    &\qquad\qquad\qquad+W_{\xi}(\varphi^2-\lambda-\varphi_{\xi})+W(\lambda_{\xi}-\varphi\lambda)\Big)\nonumber\\
    &\;+g^4\Big(6g_{\xi}(W\lambda+U_{\xi}\lambda-W_{\xi}\varphi)+4g_{\xi^2}(W_{\xi}-U\lambda)\nonumber\\
    &\qquad\qquad\qquad+g_{\xi^3}(U\varphi -W-U_{\xi})\Big)\nonumber\\
    &\; +g^3\Big(7g_{\xi}^2(W_{\xi}-U\lambda) +4g_{\xi^2}g_{\xi}(U\varphi -W-4U_{\xi})\Big)\nonumber\\
    &\; +g^2g_{\xi}^3\Big(U\varphi-W-U_{\xi}\Big),
    \end{align*}
    and
    \begin{align*}
    \frac{\bm{[}C_{p},C_{p^4t}\bm{]}}{\bm{[}C_{\xi},C_{\xi^2}\bm{]}}=\;&
    g^5\Big(U_{\xi^4}+4W_{\xi^3}-4\varphi U_{\xi^3}-6\varphi W_{\xi^2}+6U_{\xi^2}(\varphi^2-\lambda-\varphi_{\xi})\nonumber\\
                        &\qquad+4W_{\xi}(\varphi^2-\lambda-\varphi_{\xi})+W(2\lambda\varphi-\varphi^3+3\varphi\varphi_{\xi}-2\lambda_{\xi}-\varphi_{\xi^2})\nonumber\\
                        &\qquad+4U_{\xi}(2\lambda\varphi-\varphi^3+3\varphi\varphi_{\xi}-2\lambda_{\xi}-\varphi_{\xi^2})\nonumber\\
                        &\qquad+U(\varphi^4-3\lambda\varphi^2-6\varphi^2\varphi_{\xi}+\lambda^2\nonumber\\
                        &\qquad\qquad\qquad\qquad+4\lambda\varphi_{\xi}+5\lambda_{\xi}\varphi+4\varphi\varphi_{\xi^2}+3\varphi_{\xi}^2-3\lambda_{\xi^2}-\varphi_{\xi^3})\Big)\nonumber\\
                        &+g^4\Big(6g_{\xi}U_{\xi^3}+18g_{\xi}W_{\xi^2}+2(2g_{\xi^2}-9g_{\xi}\varphi)U_{\xi^2}+2W_{\xi}(4g_{\xi^2}-9g_{\xi}\varphi)\nonumber\\
                        &\qquad+U_{\xi}(18g_{\xi}\varphi^2-8g_{\xi^2}\varphi-18g_{\xi}\lambda-18g_{\xi}\varphi_{\xi}+g_{\xi^3})\nonumber\\
                        &\qquad+W(6g_{\xi}(\varphi^2-\lambda-\varphi_{\xi})-4g_{\xi^2}\varphi+g_{\xi^3})\nonumber\\
                        &\qquad+U\big(4g_{\xi^2}(\varphi^2-\lambda-\varphi_{\xi})-g_{\xi^3}\varphi\big)\nonumber\\
                        &\qquad\qquad\qquad\qquad+6g_{\xi}(2\lambda\varphi-\varphi^3+3\varphi_{\xi}\varphi-2\lambda_{\xi}-\varphi_{\xi^2})\Big)\nonumber\\
                        &+g^3\Big(7g_{\xi}^2U_{\xi^2}+14g_{\xi}^2W_{\xi}+W(4g_{\xi}g_{\xi^2}-7g_{\xi}^2\varphi)+2U_{\xi}(2g_{\xi}g_{\xi^2}-7g_{\xi}^2\varphi)\nonumber\\
                        &\qquad\qquad+U(7g_{\xi}^2\varphi^2-7g_{\xi}^2\varphi_{\xi}-7g_{\xi}^2\lambda-4g_{\xi}g_{\xi^2}\varphi)\Big)\nonumber\\
                        &+g^2g_{\xi}^3\Big(W+U_{\xi}-U\varphi\Big).
    \end{align*}
    \end{small}%
    Furthermore, we have
    \begin{align*}
    \frac{\bm{[}C_{p},C_{p^4}\bm{]}}{\bm{[}C_{\xi},C_{\xi^2}\bm{]}}
        &=g^5(\varphi^2-\lambda-\varphi_{\xi})+2g^4(2g_{\xi^2}-3g_{\xi}\varphi)+7g^3g_{\xi}^2,\\
    \frac{\bm{[}C_{pt},C_{p^2}\bm{]}}{\bm{[}C_{\xi},C_{\xi^2}\bm{]}}
        &=g^3(W_{\xi}-U\lambda)+g^2g_{\xi}(U\varphi-W-U_{\xi}),\\
    \frac{\bm{[}C_{p},C_{p^2t}\bm{]}}{\bm{[}C_{\xi},C_{\xi^2}\bm{]}}
        &=g^3\big(U(\varphi^2-\lambda-\varphi_{\xi})-\varphi(W+2U_{\xi})+2W_{\xi}+U_{\xi^2}\big) \nonumber\\
        &\qquad+g^2g_{\xi}(W+U_{\xi}-U\varphi).
    \end{align*}
    By $\bm{[}C_p,C_{p^2}\bm{]}=g^3\bm{[}C_{\xi},C_{\xi^2}\bm{]}$ and  \eqref{eq-EFG}, one can see
    \begin{small}
    \begin{align}\label{E-r}
    E=\;&g^2\bigg(U_{\xi^4}+4W_{\xi^3}-4\varphi U_{\xi^3}-6\varphi W_{\xi^2}+(5U_{\xi^2}+2W_{\xi})(\varphi^2-\varphi_{\xi}-\lambda)\nonumber\\
    &\quad\qquad\qquad+W(2\varphi_{\xi}\varphi-\varphi_{\xi^2}-\lambda_{\xi})\nonumber\\
    &\quad\qquad\qquad+U(-4\varphi_{\xi}\varphi^2+2\varphi_{\xi}^2+2\varphi_{\xi}\lambda+4\varphi\varphi_{\xi^2}+4\varphi\lambda_{\xi}-\varphi_{\xi^3}-3\lambda_{\xi^2})\nonumber\\
    &\quad\qquad\qquad+U_{\xi}(5\varphi\lambda-2\varphi^3+10\varphi_{\xi}\varphi-4\varphi_{\xi^2}-7\lambda_{\xi})\bigg)\nonumber\\
    &+6gg_{\xi}\bigg(U_{\xi^3}+3W_{\xi^2}-2\varphi U_{\xi^2}-\varphi W_{\xi}-\varphi_{\xi} W\nonumber\\
    &\;\quad\qquad\qquad+U_{\xi}(\varphi^2-3\varphi_{\xi}-2\lambda)+U(2\varphi_{\xi}\varphi-\varphi_{\xi^2}-2\lambda_{\xi})\bigg).
    \end{align}
    \end{small}%
    At the same time, it is not hard to verify
    \begin{align*}
    \frac{\bm{[}C_p,C_{p^3}\bm{]}}{\bm{[}C_p,C_{p^2}\bm{]}}=\;&3g_{\xi} - g\varphi,\\
    \frac{\bm{[}C_{pt},C_{p^3}\bm{]}}{\bm{[}C_{\xi},C_{\xi^2}\bm{]}}
    =\;&
    g^4\big(\lambda (W+U_{\xi})-W_{\xi}\varphi\big)+ g^2g_{\xi}^2(\varphi U-U_{\xi}-W) \nonumber\\
    &\quad+g^3\big(3g_{\xi}(W_{\xi}-\lambda U)+(U\varphi-U_{\xi}-W)g_{\xi^2}\big),\\
    \end{align*}
    \begin{align*}
    \frac{\bm{[}C_{p},C_{p^3t}\bm{]}}{\bm{[}C_{\xi},C_{\xi^2}\bm{]}}=\;
    &g^4\bigg(U_{\xi^3}+3W_{\xi^2}-3\varphi U_{\xi^2}+(W+3U_{\xi})(\varphi^2-\lambda-\varphi_{\xi})\nonumber\\
    &\qquad -3\varphi W_{\xi}+U(2\varphi\lambda-\varphi^3+3\varphi\varphi_{\xi}-\varphi_{\xi^2}-2\lambda_{\xi})\bigg)\nonumber\\
    &\;+g^3\bigg(3g_{\xi}(U_{\xi^2}-2U_{\xi}\varphi+2W_{\xi}-W\varphi)+(U_{\xi}+W)g_{\xi^2}\nonumber\\
    &\qquad\qquad\qquad +U\big(3g_{\xi}(\varphi^2-\lambda-\varphi_{\xi})-\varphi g_{\xi^2}\big)\bigg)\nonumber\\
    &\;+g^2g_{\xi}^2\bigg(U_{\xi}+W-U\varphi\bigg).
    \end{align*}
    Thus, it follows by \eqref{eq-EFG}
    \begin{small}
    \begin{align}\label{F-r}
    F\;=&g(g\varphi - 3g_{\xi})\bigg(-U_{\xi^3}+2\varphi U_{\xi^2}-3W_{\xi^2}+\varphi W_{\xi}+\varphi_{\xi} W\nonumber\\
    &\qquad\qquad\qquad+U_{\xi}(2\lambda-\varphi^2+3\varphi_{\xi})+U(2\lambda_{\xi}-2\varphi\varphi_{\xi}+\varphi_{\xi^2})\bigg).
    \end{align}
    \end{small}%
    Again, by
    \begin{small}
    \begin{align*}
     \frac{\bm{[}C_{p^2t},C_{p^3}\bm{]}}{\bm{[}C_{\xi},C_{\xi^2}\bm{]}}=\;
     &g^5\bigg(\lambda U_{\xi^2}-\varphi W_{\xi^2}+2\lambda W_{\xi}+U(\lambda_{\xi}\varphi-\lambda^2-\lambda\varphi_{\xi})\bigg)\nonumber\\
     &+g^4\bigg(3g_{\xi}W_{\xi^2}-g_{\xi^2}U_{\xi^2}+U_{\xi}(2\varphi g_{\xi^2}-5\lambda g_{\xi})+W_{\xi}(-g_{\xi}\varphi-2g_{\xi^2})\nonumber\\
     &\qquad+U\big(3g_{\xi}(\lambda \varphi-\lambda_{\xi})+g_{\xi^2}(\lambda +\varphi_{\xi}-\varphi^2)\big)+W(\varphi g_{\xi^2}-2\lambda g_{\xi})\bigg)\nonumber\\
     &+g^3g_{\xi}\bigg(g_{\xi}(W_{\xi}-U_{\xi^2})+U_{\xi}(2g_{\xi}\varphi-g_{\xi^2})+W(g_{\xi}\varphi-g_{\xi^2})\nonumber\\
     &\qquad+U\big(g_{\xi}(\varphi_{\xi}-\varphi^2-2\lambda)+\varphi g_{\xi^2}\big)\bigg)\nonumber\\
     &+g^2g_{\xi}^3\bigg(U\varphi-U_{\xi}-W\bigg),
    \end{align*}
    \begin{align*}
     \frac{\bm{[}C_{p^2},C_{p^3t}\bm{]}}{\bm{[}C_{\xi},C_{\xi^2}\bm{]}}=\;
     &g^5\bigg(-W_{\xi^3}+3\lambda U_{\xi^2}+3\lambda W_{\xi}+U(\varphi^2\lambda-\varphi\lambda_{\xi}-\lambda^2-2\lambda\varphi_{\xi}+\lambda_{\xi^2})\nonumber\\
     &\quad+(3U_{\xi}+W)(\lambda_{\xi}-\varphi\lambda)\bigg)\nonumber\\
     &+g^4\bigg(g_{\xi}(U_{\xi^3}-3\varphi U_{\xi^2})+W_{\xi}(-3g_{\xi}\varphi-g_{\xi^2})+3U_{\xi}g_{\xi}(\varphi^2+\lambda-\varphi_{\xi})\nonumber\\
     &\qquad+Wg_{\xi}(\varphi^2+2\lambda-\varphi_{\xi})\nonumber\\
     &\qquad+U(-g_{\xi}\varphi^3-\lambda g_{\xi}\varphi+3g_{\xi}\varphi\varphi_{\xi}+g_{\xi}\lambda_{\xi}-g_{\xi}\varphi_{\xi^2}+\lambda g_{\xi^2})\bigg)\nonumber\\
     &+g^3g_{\xi}\bigg(3g_{\xi}U_{\xi^2}+U_{\xi}(-6g_{\xi}\varphi+g_{\xi^2})+5g_{\xi}W_{\xi}+W(-3g_{\xi}\varphi+g_{\xi^2})\nonumber\\
     &\qquad+U\big(g_{\xi}(3\varphi^2-2\lambda-3\varphi_{\xi})-\varphi g_{\xi^2}\big)\bigg)\nonumber\\
     &+g^2g_{\xi}^3\bigg(W+U_{\xi}-U\varphi\bigg),
    \end{align*}
    \end{small}%
    and
    \begin{align*}
     \frac{\bm{[}C_{p^2},C_{p^3}\bm{]}}{\bm{[}C_{\xi},C_{\xi^2}\bm{]}}=g^5\lambda +g^4(-g_{\xi}\varphi-g_{\xi^2})+2g_{\xi}^2g^3,
     \end{align*}
    in view of \eqref{eq-EFG}, one can obtain
    \begin{small}
    \begin{align}\label{G-r}
    G=\;&gg_{\xi}\bigg(U_{\xi^3}+3W_{\xi^2}-\varphi (2U_{\xi^2}+ W_{\xi})\nonumber\\
    &\qquad\qquad\qquad+U_{\xi}(\varphi^2-2\lambda -3\varphi_{\xi})-\varphi_{\xi}W+U(2\varphi\varphi_{\xi}-2\lambda_{\xi}-\varphi_{\xi^2})\bigg)\nonumber\\
    &\;+g^2\bigg(-W_{\xi^3}+3\lambda U_{\xi^2}-\varphi W_{\xi^2}\nonumber\\
    &\qquad\qquad\qquad+U_{\xi}(3\lambda_{\xi}-\varphi\lambda)+2\lambda W_{\xi}+\lambda_{\xi}W+U(\lambda_{\xi^2}-2\lambda\varphi_{\xi})\bigg).
    \end{align}
    \end{small}%
    Hence, combining of \eqref{gt_EFG}, \eqref{E-r}, \eqref{F-r} and \eqref{G-r} leads to
    \begin{small}
    \begin{align*}
    2\epsilon g\frac{\partial g}{\partial t}=\;&g^2\bigg(3U_{\xi^4}-2\varphi U_{\xi^3}+U_{\xi^2}(21\lambda-5\varphi^2-15\varphi_{\xi})\nonumber\\
    &\qquad+U_{\xi}(4\varphi^3-17\varphi\lambda+15\lambda_{\xi}-12\varphi_{\xi^2})\nonumber\\
    &\qquad+2W_{\xi}(9\lambda-2\varphi^2-3\varphi_{\xi})+W(9\lambda_{\xi}-4\varphi\varphi_{\xi}-3\varphi_{\xi^2})\nonumber\\
    &\qquad+U(8\varphi^2\varphi_{\xi}-18\lambda\varphi_{\xi}-8\varphi\lambda_{\xi}+2\varphi\varphi_{\xi^2}+6\varphi_{\xi}^2+3\lambda_{\xi^2}-3\varphi_{\xi^3})\bigg),
    \end{align*}
    \end{small}%
    or, substituting \eqref{lam_exp} into this equation, equivalently,
    \begin{small}
    \begin{align}\label{gtog}
    \epsilon\frac{\partial g}{\partial t}=\frac{g}{18}&\bigg(27U_{\xi^4}-18\varphi U_{\xi^3}+U_{\xi^2}(21\epsilon-3\varphi^2-72\varphi_{\xi})\nonumber\\
    &\;\;+U_{\xi}(2\varphi^3+9\varphi\varphi_{\xi}-17\epsilon\varphi-63\varphi_{\xi^2})\nonumber\\
    &\;\;+18\epsilon W_{\xi}+U(4\varphi^2\varphi_{\xi}+6\varphi\varphi_{\xi^2}+12\varphi_{\xi}^2-18\epsilon\varphi_{\xi}-18\varphi_{\xi^3})\bigg).
    \end{align}
    \end{small}%
    To make further progress, we now come to a crucial computation, namely the  fully affine curvature evolution. By \eqref{affine-curv}, we get
    \begin{small}
    \begin{align}\label{varphi-t-origin}
          \begin{aligned}
            \frac{\partial \varphi}{\partial t}&          =-\frac{\partial}{\partial t}\frac{\bm{[}C_{\xi},C_{\xi^3}\bm{]}}{\bm{[}C_{\xi},C_{\xi^2}\bm{]}}\\
              &=-\frac{
                        \left(\bm{[}C_{\xi t},C_{\xi^3}\bm{]}+\bm{[}C_{\xi},C_{\xi^3 t}\bm{]}\right)\bm{[}C_{\xi},C_{\xi^2}\bm{]}
                        -\bm{[}C_{\xi},C_{\xi^3}\bm{]}\left(\bm{[}C_{\xi t},C_{\xi^2}\bm{]}+\bm{[}C_{\xi},C_{\xi^2 t}\bm{]}\right)
                    }{
                        \bm{[}C_{\xi},C_{\xi^2}\bm{]}^2
                    }.
          \end{aligned}
    \end{align}
    \end{small}%
    Again recalling \eqref{req-eq},
    we have, in view of \eqref{c3xi} and \eqref{curve-motion}
    \begin{small}
    \begin{align*}
    C_{\xi t}&=C_{t\xi}-\frac{g_t}{g}C_{\xi}\nonumber\\
                   &=(WC_{\xi}+UC_{\xi^2})_{\xi}-\frac{g_t}{g}C_{\xi}\nonumber\\
                   &=(W_{\xi}+U\lambda-\frac{g_t}{g})C_{\xi}+(W+U_{\xi}-U\varphi)C_{\xi^2},\\
    C_{\xi^2 t}&=C_{\xi t\xi}-\frac{g_t}{g}C_{\xi^2}\nonumber\\
                    &=\Big(-\big(\frac{g_t}{g}\big)_{\xi} - W\lambda + U(\varphi\lambda- \lambda_{\xi}) - 2\lambda U_{\xi} + W_{\xi^2} \Big)C_{\xi}\nonumber\\
                    &\qquad+\Big(-2\frac{g_t}{g} + U(\varphi^2 - \lambda - \varphi_{\xi}) - (W+ 2 U_{\xi})\varphi   + 2W_{\xi}   + U_{\xi^2}\Big)C_{\xi^2},\\
    C_{\xi^3 t}&=C_{\xi^2 t\xi}-\frac{g_t}{g}C_{\xi^3}\nonumber\\
                    &=\Big(W_{\xi^3}- 3\lambda U_{\xi^2}+ 3U_{\xi}(\varphi\lambda  - \lambda_{\xi}) +U(\lambda^2 - \lambda_{\xi^2}- \varphi^2\lambda + \varphi \lambda_{\xi} + 2\lambda \varphi_{\xi})    \nonumber\\
                    &\qquad\qquad    - 3\lambda W_{\xi}+ W(\varphi\lambda- \lambda_{\xi})  + 3\frac{g_t}{g}\lambda-\big(\frac{g_t}{g}\big)_{\xi^2}\Big)C_{\xi}\nonumber\\
                    &\qquad+\Big(U_{\xi^3}- 3\varphi U_{\xi^2}+3U_{\xi}(\varphi^2 - \varphi_{\xi}- \lambda) + U(2\varphi\lambda- \varphi^3 - 2\lambda_{\xi}+ 3\varphi\varphi_{\xi}- \varphi_{\xi^2}  )   \nonumber\\
                    &\qquad \qquad   + 3W_{\xi^2}- 3\varphi W_{\xi}+ W(\varphi^2- \lambda - \varphi_{\xi})   + 3\frac{g_t}{g}\varphi -3\big(\frac{g_t}{g}\big)_{\xi}\Big)C_{\xi^2}.
    \end{align*}
    \end{small}%
    Thus, by \eqref{c3xi}, \eqref{varphi-t-origin} can be represented as
    \begin{small}
    \begin{align*}
    \frac{\partial\varphi}{\partial t}=\;&
    \frac{9}{2}\epsilon U_{\xi^5}- \frac{9}{2}\epsilon\varphi U_{\xi^4}+ U_{\xi^3}\Big(\frac{63}{2}\epsilon\lambda- \frac{13}{2}\epsilon\varphi^2 - \frac{51}{2}\epsilon\varphi_{\xi}-1\Big)  \nonumber\\
    &+ U_{\xi^2}\Big(\frac{17}{2}\epsilon\varphi^3+2\varphi- 36\epsilon\varphi\lambda+  54\epsilon\lambda_{\xi}- \frac{15}{2}\epsilon\varphi\varphi_{\xi}- \frac{81}{2}\epsilon \varphi_{\xi^2} \Big) \nonumber\\
    &+U_{\xi}\Big(- 2\epsilon\varphi^4- \varphi^2+ \frac{17}{2}\epsilon\varphi^2\lambda+ 3\varphi_{\xi} + 2\lambda + 9\epsilon \varphi_{\xi}^2+ 30\epsilon\varphi^2 \varphi_{\xi}- 45\epsilon\varphi \lambda_{\xi}\nonumber\\
    & \qquad\qquad\qquad\qquad- \frac{105}{2}\epsilon\lambda \varphi_{\xi}+ 9\epsilon\varphi \varphi_{\xi^2}+ 27\epsilon \lambda_{\xi^2}- \frac{45}{2}\epsilon \varphi_{\xi^3}\Big)\nonumber\\
    &+ U\Big(- 39\epsilon \varphi_{\xi}\lambda_{\xi}+ 9\epsilon \varphi\lambda\varphi_{\xi}+ 21\epsilon \varphi\varphi_{\xi}^2 - 4\epsilon \varphi^3\varphi_{\xi} +4\epsilon \varphi^2\lambda_{\xi}   + 2\lambda_{\xi}\nonumber\\
    &\qquad\qquad- 2\varphi\varphi_{\xi}+\varphi_{\xi^2}+ 21\epsilon \varphi_{\xi}\varphi_{\xi^2}- \frac{27}{2}\epsilon \varphi\lambda_{\xi^2} - 27\epsilon \lambda\varphi_{\xi^2}\nonumber\\
    &\qquad\qquad+ 11\epsilon \varphi^2\varphi_{\xi^2}+ \frac{9}{2}\epsilon \varphi\varphi_{\xi^3}+ \frac{9}{2}\epsilon \lambda_{\xi^3}- \frac{9}{2}\epsilon \varphi_{\xi^4}\Big) \nonumber\\
    & + W\Big(\varphi_{\xi}- 6\epsilon \varphi_{\xi}^2- \frac{9}{2}\epsilon \varphi\lambda_{\xi}+ 2\epsilon \varphi^2\varphi_{\xi}+ \frac{27}{2}\epsilon \lambda_{\xi^2}- \frac{9}{2}\epsilon \varphi\varphi_{\xi^2}- \frac{9}{2}\epsilon \varphi_{\xi^3}\Big)\nonumber\\
    &   + W_{\xi}\Big(2\epsilon\varphi^3+\varphi- 9\epsilon\varphi\lambda+ \frac{81}{2}\epsilon \lambda_{\xi} - 15\epsilon\varphi \varphi_{\xi}- \frac{27}{2}\epsilon \varphi_{\xi^2}\Big) \nonumber\\
     &+W_{\xi^2}\Big(27\epsilon\lambda- 3 - 6\epsilon\varphi^2  - 9\epsilon\varphi_{\xi}       \Big) .
    \end{align*}
    \end{small}%
    Substituting \eqref{lam_exp} into above equation generates
    \begin{small}
    \begin{align}\label{vaphi-evo}
    \frac{\partial\varphi}{\partial t}=\;&W\varphi_{\xi}+\frac{9}{2}\epsilon U_{\xi^5}-\frac{9}{2}\epsilon\varphi U_{\xi^4}+U_{\xi^3}\Big(\frac{\epsilon}{2}\varphi^2-15\epsilon\varphi_{\xi} +\frac{5}{2}\Big)\nonumber\\
    &+U_{\xi^2}\Big(\frac{\epsilon}{2}\varphi^3-2\varphi+\frac{9}{2}\epsilon\varphi\varphi_{\xi}-\frac{45}{2}\epsilon \varphi_{\xi^2}\Big)\nonumber\\
    &+U_{\xi}\Big(-\frac{\epsilon}{9}\varphi^4+\frac{7}{18}\varphi^2+\frac{2}{9}\epsilon+\frac{7}{6}\epsilon\varphi^2\varphi_{\xi}-\frac{13}{6}\varphi_{\xi}+\frac{7}{2}\epsilon \varphi^2_{\xi}  +6\epsilon\varphi \varphi_{\xi^2}-\frac{27}{2}\epsilon \varphi_{\xi^3}\Big)  \nonumber\\
    &+U\Big(-\frac{2}{9}\epsilon \varphi^3\varphi_{\xi}+\frac{2}{3}\epsilon \varphi\varphi^2_{\xi}-\frac{1}{9}\varphi\varphi_{\xi}+\frac{\epsilon}{3}\varphi^2\varphi_{\xi^2}\nonumber\\
    &\qquad\qquad\qquad+5\epsilon \varphi_{\xi}\varphi_{\xi^2}-\frac{4}{3}\varphi_{\xi^2}+2\epsilon \varphi\varphi_{\xi^3}
    -3\epsilon \varphi_{\xi^4}\Big).
    \end{align}
    \end{small}%
 \subsection{The evolution of higher order derivatives}
    By \eqref{gtog} and \eqref{vaphi-evo}, we can calculate the evolutions of the higher order derivatives for the curvature $\varphi$ with respect to the fully affine arc length parameter $\xi$.
    \begin{small}
    \begin{align}\label{varphi1t}
      \varphi_{\xi t}=\;&\varphi_{t\xi}-\frac{g_t}{g}\varphi_{\xi}\nonumber\\
      =\;&\varphi_{\xi^2}W+\frac{9}{2}\epsilon (U_{\xi^6}-\varphi U_{\xi^5})+U_{\xi^4}\big(\frac{\epsilon\varphi^2}{2}-21\epsilon\varphi_\xi+\frac{5}{2}\big)\nonumber\\
      &\quad+U_{\xi^3}\frac{\varphi(\epsilon\varphi^2+13\epsilon\varphi_{\xi}-4)-75\epsilon\varphi_{\xi^2}}{2}\nonumber\\
      &\quad+U_{\xi^2}\Big(-36\epsilon\varphi_{\xi^3}+\frac{21\epsilon\varphi\varphi_{\xi^2}}{2}+12\epsilon\varphi_{\xi}^2\nonumber\\
      &\qquad\qquad\qquad+(3\epsilon\varphi^2-\frac{16}{3})\varphi_{\xi}-\frac{2\epsilon\varphi^4-7\varphi^2+4\epsilon}{18}\Big)\nonumber\\
      &\quad+U_{\xi}\Big(-\frac{33}{2}\epsilon\varphi_{\xi^4}+8\epsilon\varphi\varphi_{\xi^3}\nonumber\\
      &\qquad\qquad\qquad+\frac{3\epsilon\varphi^2+43\epsilon\varphi_{\xi}-7}{2}\varphi_{\xi^2}+\frac{45\epsilon\varphi_{\xi}-14\epsilon\varphi^2+29}{18}\varphi\varphi_{\xi}\Big)\nonumber\\
      &\quad+U\Big(-3\epsilon\varphi_{\xi^5}+2\epsilon\varphi\varphi_{\xi^4}+\frac{\epsilon\varphi^2+24\epsilon\varphi_{\xi}-4}{3}\varphi_{\xi^3}\nonumber\\
      &\qquad\qquad\qquad+\frac{-2\epsilon\varphi^2+15\epsilon\varphi_{\xi}-1}{9}\varphi\varphi_{\xi^2}+5\epsilon\varphi_{\xi^2}^2-\frac{8}{9}(\epsilon\varphi^2-1)\varphi_{\xi}^2\Big).
    \end{align}
    \begin{align}\label{varphi2t}
    \varphi_{\xi^2 t}=\;&\varphi_{\xi t\xi}-\frac{g_t}{g}\varphi_{\xi^2}\nonumber\\
    =\;&\varphi_{\xi^3}W+\frac{9}{2}\epsilon(U_{\xi^7}-\varphi U_{\xi^6})+U_{\xi^5}\frac{\epsilon\varphi^2-51\varphi_{\xi}+5}{2}\nonumber\\
    &\quad+U_{\xi^4}\bigg(\frac{\epsilon\varphi^2+15\epsilon\varphi_{\xi}-4}{2}\varphi-60\epsilon\varphi_{\xi^2}\bigg)\nonumber\\
    &\quad+U_{\xi^3}\bigg(-\frac{147\epsilon\varphi_{\xi^3}}{2}+18\epsilon\varphi\varphi_{\xi^2}\nonumber\\
    &\qquad\qquad\qquad+\frac{37\epsilon\varphi_{\xi}^2}{2}+\frac{13\epsilon\varphi^2-22}{3}\varphi_{\xi}-\frac{2\epsilon\varphi^4-7\varphi^2-4\epsilon}{18}\bigg)\nonumber\\
    &\quad+U_{\xi^2}\bigg(-\frac{105}{2}\epsilon\varphi_{\xi^4}+\frac{37}{2}\epsilon\varphi_{\xi^3}\nonumber\\
    &\qquad\qquad\qquad+\frac{9\epsilon\varphi^2+120\epsilon\varphi_{\xi}-20}{2}\varphi_{\xi^2}-\frac{22\epsilon\varphi^2-147\epsilon\varphi_{\xi}-43}{18}\varphi_{\xi}\varphi\bigg)\nonumber\\
    &\quad+U_{\xi}\bigg(-\frac{39}{2}\epsilon\varphi_{\xi^5}+10\epsilon\varphi\varphi_{\xi^4}+\frac{33\epsilon\varphi^2+675\epsilon\varphi_{\xi}-87}{18}\varphi_{\xi^3}+30\epsilon\varphi_{\xi^2}^2\nonumber\\
    &\qquad\qquad\qquad+\frac{-20\epsilon\varphi^2+165\epsilon\varphi\varphi_{\xi}+44}{18}\varphi\varphi_{\xi^2}+\frac{45\epsilon\varphi_{\xi}-58\epsilon\varphi^2+45}{18}\varphi_{\xi}^2\bigg)\nonumber\\
    &\quad+U\bigg(-3\epsilon\varphi_{\xi^6}+2\epsilon\varphi\varphi_{\xi^5}+\frac{\epsilon\varphi^2+30\epsilon\varphi_{\xi}-4}{3}\varphi_{\xi^4}\nonumber\\
    &\qquad\qquad\qquad+\frac{-4\epsilon\varphi^3+21\epsilon\varphi\varphi_{\xi}+171\epsilon\varphi_{\xi^2}-\varphi}{9}\varphi_{\xi^3}\nonumber\\
    &\qquad\qquad\qquad+\frac{4}{3}\epsilon\varphi\varphi_{\xi^2}^2+\frac{3\epsilon\varphi_{\xi}-8\epsilon\varphi^2+8}{3}\varphi_{\xi}\varphi_{\xi^2}
    -\frac{16}{9}\epsilon\varphi\varphi_{\xi}^3\bigg).
    \end{align}
    \begin{align}\label{varphi3t}
    \varphi_{\xi^3 t}=\;&\varphi_{\xi^2 t\xi}-\frac{g_t}{g}\varphi_{\xi^3}\nonumber\\
    =\;&\varphi_{\xi^4}W+\frac{9}{2}\epsilon(U_{\xi^8}-\varphi U_{\xi^7})+U_{\xi^6}\frac{\epsilon\varphi^2-60\epsilon\varphi_{\xi}+5}{2}\nonumber\\
    &\;\;\;+U_{\xi^5}\frac{\varphi(\epsilon\varphi^2+17\epsilon\varphi_{\xi}-4)-171\epsilon\varphi_{\xi^2}}{2}\nonumber\\
    &\;\;\;+U_{\xi^4}\bigg(-135\epsilon\varphi_{\xi^3}+\frac{51}{2}\epsilon\varphi\varphi_{\xi^2}+26\epsilon\varphi_{\xi}^2
    +\frac{35\epsilon\varphi^2-56}{6}\varphi_{\xi}\nonumber\\
    &\qquad\qquad\qquad\qquad\qquad\qquad\qquad\qquad\qquad-\frac{2\epsilon\varphi^4-7\varphi^2-4\epsilon}{18}\bigg)\nonumber\\
    &\;\;\;+U_{\xi^3}\bigg(-126\epsilon\varphi_{\xi^4}+\frac{75}{2}\epsilon\varphi\varphi_{\xi^3}+\frac{53\epsilon\varphi^2+690\epsilon\varphi_{\xi}-104}{6}\varphi_{\xi^2}\nonumber\\
    &\qquad\qquad\qquad\qquad\qquad\qquad-\frac{10\epsilon\varphi^2-101\epsilon\varphi_{\xi}-19}{6}\varphi\varphi_{\xi}\bigg)\nonumber\\
    &\;\;\;+U_{\xi^2}\bigg(-72\epsilon\varphi_{\xi^5}+\frac{57}{2}\epsilon\varphi\varphi_{\xi^4}+\frac{39\epsilon\varphi^2-96+720\epsilon\varphi_{\xi}}{6}\varphi_{\xi^3}
    +90\epsilon\varphi_{\xi^2}^2\nonumber\\
    &\qquad\qquad\qquad+\frac{32}{3}\epsilon\varphi_{\xi}^3+\frac{-14\epsilon\varphi^2+207\epsilon\varphi_{\xi}+29}{6}\varphi\varphi_{\xi^2}+\frac{44-62\epsilon\varphi^2}{9}\varphi_{\xi}^2\bigg)\nonumber\\
    &\;\;\;+U_{\xi}\bigg(-\frac{45}{2}\epsilon\varphi_{\xi^6}+12\epsilon\varphi\varphi_{\xi^5}+\frac{13\epsilon\varphi^2+345\epsilon\varphi_{\xi}-37}{6}\varphi_{\xi^4}
    +\frac{21}{2}\epsilon\varphi\varphi_{\xi^2}^2\nonumber\\
    &\qquad\qquad\quad+\frac{-26\epsilon\varphi^3+264\epsilon\varphi\varphi_{\xi}+2160\epsilon\varphi_{\xi^2}+58\varphi}{18}\varphi_{\xi^3}\nonumber\\
    &\qquad\qquad\quad+\frac{159\epsilon\varphi_{\xi}-112\epsilon\varphi^2+91}{9}\varphi_{\xi}\varphi_{\xi^2}-\frac{74}{9}\epsilon\varphi\varphi_{\xi}^3\bigg)\nonumber\\
    &\quad+U\bigg(-3\epsilon\varphi_{\xi^7}+2\epsilon\varphi_{\xi^6}+\frac{\epsilon\varphi^2+36\epsilon\varphi_{\xi}-4}{3}\varphi_{\xi^5}
    +20\epsilon\varphi_{\xi^3}^2-\frac{32}{3}\epsilon\varphi\varphi_{\xi}^2\varphi_{\xi^2}\nonumber\\
    &\qquad\qquad\qquad+\frac{-2\epsilon\varphi^3+27\epsilon\varphi\varphi_{\xi}+261\epsilon\varphi_{\xi^2}-\varphi}{9}\varphi_{\xi^4}-\frac{16}{9}\epsilon\varphi_{\xi}^4\nonumber\\
    &\qquad\qquad\qquad+\frac{-8\epsilon\varphi^2+10\epsilon\varphi_{\xi}+8}{3}\varphi_{\xi^2}^2\nonumber\\
    &\qquad\qquad\qquad+\frac{42\epsilon\varphi\varphi_{\xi^2}+24\epsilon\varphi_{\xi}^2+32(1-\epsilon\varphi^2)\varphi_{\xi}}{9}\varphi_{\xi^3}\bigg).
    \end{align}
    \end{small}%
\section{The stability of the extremal curves on fully affine plane}\label{sec-VF}

    The   fully affine length of a curve $C:[0,1]\rightarrow\R^2$
    \begin{equation*}
      l:=\int^1_0g dp
    \end{equation*}
    can be calculated. It is of interest to consider how the functional $l$ varies when $C$ is deformed smoothly.

    Assume $C=C(\xi,t), \quad 0\leq \xi \leq 1, \quad 0\leq t<T,$  is a smooth deformation of $C_0=C(\cdot,0)$
    such that
    \begin{equation}\label{WU-flow}
    C_t=WC_\xi+UC_{\xi^2},
    \end{equation}
    where $W$ and $U$ are arbitrary smooth functions of $\xi$ and $t$. Clearly, the first variation of fully affine arc length is
    \begin{equation*}
    l'(t)=\int^1_0g_t dp.
    \end{equation*}
    By \eq{gtog}, if $W(1)=W(0)=0$, $U_{p^i}(1)=U_{p^i}(0)=0, i=1,2,3$, $U(1)=U(0)=0$,
    then the  first variation formula can be written as
    \begin{small}
    \begin{align}\label{First-VF}
    l'(t)&=\frac{\epsilon}{18}\int^1_0\bigg(27U_{\xi^4}-18\varphi U_{\xi^3}+U_{\xi^2}(21\epsilon-3\varphi^2-72\varphi_{\xi})\nonumber\\
             &\qquad\qquad\qquad\;+U_{\xi}(2\varphi^3+9\varphi\varphi_{\xi}-17\epsilon\varphi-63\varphi_{\xi^2})+18\epsilon W_{\xi}\nonumber\\
              &\qquad\qquad\qquad\;+U(4\varphi^2\varphi_{\xi}+6\varphi\varphi_{\xi^2}+12\varphi_{\xi}^2-18\epsilon\varphi_{\xi}-18\varphi_{\xi^3})\bigg)gdp\nonumber\\
         &=\frac{\epsilon}{18}\int^1_0U\bigg(18\varphi_{\xi^3}-(6\varphi_\xi^2+6\varphi\varphi_{\xi^2}+72\varphi_{\xi^3})\nonumber\\
         &\qquad\qquad\qquad\;-(6\varphi^2\varphi_{\xi}+9\varphi_{\xi}^2+9\varphi\varphi_{\xi^2}-17\epsilon\varphi_{\xi}-63\varphi_{\xi^3})\nonumber\\
         &\qquad\qquad\qquad\;+(4\varphi^2\varphi_{\xi}+6\varphi\varphi_{\xi^2}+12\varphi_{\xi}^2-18\epsilon\varphi_{\xi}-18\varphi_{\xi^3})\bigg)gdp\nonumber\\
         &=-\frac{\epsilon}{2}\int^1_0U\bigg(\varphi_{\xi^3}+\varphi\varphi_{\xi^2}+\frac{2\varphi^2+3\varphi_{\xi}+\epsilon}{9}\bigg)gdp.
   \end{align}
   \end{small}%
    Owing to the arbitrariness of function $U$, we have the following result (this result was also mentioned in \cite{ks} and \cite{mih}).
    \begin{thm}
     A plane curve  is fully affine extremal relative to the length functional if and only if
    \begin{align*}
    \varphi_{\xi^3}+\varphi\varphi_{\xi^2}+\frac{2\varphi^2+3\varphi_{\xi}+\epsilon}{9}\varphi_{\xi}=0
    \end{align*}
    holds. In particular, any curves of constant fully affine curvature are extremal.
    \end{thm}
    We now turn to the main task at hand --- calculating the second variation formula. If the curve is extremal at $t=t_0$, we have
    \begin{align}\label{extm-eq}
       \varphi_{\xi^3}(\xi,t_0)+\varphi(\xi,t_0)\varphi_{\xi^2}(\xi,t_0)+\frac{2\varphi(\xi,t_0)^2+3\varphi_{\xi}(\xi,t_0)+\epsilon}{9}\varphi_{\xi}(\xi,t_0)=0.
    \end{align}
    It is worth pointing out that in the subsequent calculations, \eqref{extm-eq} shall be repeatedly applied to remove the higher derivatives $\varphi_{\xi^i}, \;i\geq3$ at $t=t_0$.
    Then by using \eqref{vaphi-evo}, \eqref{varphi1t}, \eqref{varphi2t}, \eqref{varphi3t}, \eqref{First-VF} and \eqref{extm-eq}, we find
    \begin{align}\label{sec-f}
    l''(t_0)&=-\frac{\epsilon}{2}\int^1_0U\bigg(\varphi_{\xi^3 t}+\varphi_t\varphi_{\xi^2}+\varphi\varphi_{\xi^2 t}+\frac{2\varphi^2+3\varphi_{\xi}+\epsilon}{9}\varphi_{\xi t}\nonumber\\
    &\qquad\qquad\qquad\qquad\qquad\qquad\qquad+\frac{4\varphi\varphi_t+3\varphi_{\xi t}}{9}\varphi_{\xi}\bigg)gdp.\nonumber\\
    &=-\frac{1}{2}\int^1_0U\big(\frac{9}{2}U_{\xi^8}+f_6U_{\xi^6}+f_5U_{\xi^5}+f_4U_{\xi^4}+f_3U_{\xi^3}\nonumber\\
    &\qquad\qquad\qquad\qquad\qquad\qquad\qquad+f_2U_{\xi^2}+f_1U_{\xi}+f_0U\big)gdp.
    \end{align}
    where
    \begin{small}
    \begin{align*}
    f_0&=\frac{2}{27}\varphi_{\xi}^2+6\varphi_{\xi}^4+\frac{196}{27}\epsilon\varphi_{\xi}^3-2\epsilon\varphi_{\xi^2}^2+\frac{8}{9}\varphi^4\varphi_{\xi}^2-100\varphi_{\xi}\varphi_{\xi^2}^2+\frac{152}{3}\varphi_{\xi^2}^2\varphi^2
    +\frac{16}{27}\epsilon\varphi_{\xi}^2\varphi^2\\
    &\qquad+\frac{152}{9}\varphi^2\varphi_{\xi}^3+\frac{112}{9}\epsilon\varphi\varphi_{\xi^2}\varphi_{\xi}+\frac{88}{3}\varphi\varphi_{\xi}^2\varphi_{\xi^2}+\frac{176}{9}\varphi_{\xi^2}\varphi_{\xi}\varphi^3,\\
    f_1&=\frac{4}{9}\varphi_{\xi^2}-\frac{8}{27}\varphi^5\varphi_{\xi}-190\varphi\varphi_{\xi^2}^2-200\varphi_{\xi^2}\varphi_{\xi}^2-\frac{122}{3}\varphi_{\xi}^3\varphi+\frac{136}{9}\varphi_{\xi}^2\varphi^3-\frac{8}{9}\varphi_{\xi^2}\varphi^4\\
    &\qquad+\frac{116}{9}\epsilon\varphi\varphi_{\xi}^2-\frac{2}{27}\varphi\varphi_{\xi}-\frac{8}{27}\epsilon\varphi^3\varphi_{\xi}-24\epsilon\varphi_{\xi}\varphi_{\xi^2}+\frac{4}{9}\epsilon\varphi^2\varphi_{\xi^2}-16\varphi^2\varphi_{\xi}\varphi_{\xi^2},\\
    f_2&=\frac{8}{9}\varphi_{\xi}+\frac{2}{27}\epsilon\varphi^4-\frac{76}{3}\epsilon\varphi_{\xi}^2-\frac{2}{81}\varphi^6+\frac{5}{54}\varphi^2+\frac{375}{2}\varphi_{\xi^2}^2-\frac{136}{3}\varphi_{\xi}^3+\frac{28}{3}\varphi_{\xi^2}\varphi^3\\
    &\qquad-\frac{145}{2}\varphi_{\xi}^2\varphi^2+\frac{26}{9}\varphi^4\varphi_{\xi}+\frac{47}{9}\epsilon\varphi^2\varphi_{\xi}+\frac{2}{3}\epsilon\varphi\varphi_{\xi^2}-176\varphi\varphi_{\xi}\varphi_{\xi^2}+\frac{2}{81}\epsilon,\\
    f_3&=-5\epsilon\varphi_{\xi^2}-43\varphi_{\xi^2}\varphi^2+285\varphi_{\xi}\varphi_{\xi^2}-\frac{46}{3}\varphi^3\varphi_{\xi}+59\varphi\varphi_{\xi}^2-\frac{41}{3}\epsilon\varphi\varphi_{\xi},\\
    f_4&=5\epsilon\varphi_{\xi}-\epsilon\varphi^2+57\varphi_{\xi}^2+\frac{\varphi^4}{2}+96\varphi\varphi_{\xi^2}+37\varphi^2\varphi_{\xi}+\frac{1}{2},\\
    f_5&=-9(2\varphi\varphi_{\xi}+9\varphi_{\xi^2}),\\
    f_6&=-3(\varphi^2+9\varphi_{\xi}-\epsilon).
    \end{align*}
    \end{small}%
    Recalling  $U_{p^i}(1)=U_{p^i}(0)=0, i=1,2,3$, $U(1)=U(0)=0$, integration by parts generates
    \begin{align*}
        \frac{9}{2}\int^1_0UU_{\xi^8}gdp&=\;\frac{9}{2}\int^1_0U_{\xi^4}^2gdp,\\
        \int^1_0f_6UU_{\xi^6}gdp&=\;\int^1_0\Big(-f_6U_{\xi^3}^2+\frac{9}{2}(f_6)_{\xi^2}U_{\xi^2}^2\nonumber\\
        &\qquad\qquad\qquad-3(f_6)_{\xi^4}U_{\xi}^2+\frac{1}{2}U^2(f_6)_{\xi^6}\Big)gdp,\\
        \int^1_0f_5UU_{\xi^5}gdp&=\;\int^1_0\Big(-\frac{5}{2}(f_5)_{\xi}U_{\xi^2}^2+\frac{5}{2}(f_5)_{\xi^3}U_{\xi}^2-\frac{1}{2}(f_5)_{\xi^5}U^2\Big)gdp,\\
        \int^1_0f_4UU_{\xi^4}gdp&=\;\int^1_0\Big(f_4U_{\xi^2}^2-2(f_4)_{\xi^2}U_{\xi}^2+\frac{1}{2}(f_4)_{\xi^4}U^2\Big)gdp,\\
        \int^1_0f_3UU_{\xi^3}gdp&=\;\int^1_0\Big(\frac{3}{2}(f_3)_{\xi}U_{\xi}^2-\frac{1}{2}(f_3)_{\xi^3}U^2\Big)gdp,\\
        \int^1_0f_2UU_{\xi^2}gdp&=\;\int^1_0\Big(-f_2U_{\xi}^2+\frac{1}{2}(f_2)_{\xi^2}U^2\Big)gdp,\\
        \int^1_0f_1UU_{\xi}gdp&=\;\int^1_0\Big(-\frac{1}{2}(f_1)_{\xi}U^2\Big)gdp.
    \end{align*}
    Hence, it is easy to check directly that,
    \begin{align}
        l''(t_0)&=-\frac{1}{2}\int^1_0\big(\frac{9}{2}U_{\xi^4}^2+P_3U_{\xi^3}^2+P_2U_{\xi^2}^2+P_1U_{\xi}^2+P_0U^2\big)gdp,
    \end{align}
    where
    \begin{small}
    \begin{align*}
        P_0&=\frac{1}{2}\big((f_6)_{\xi^6}-(f_5)_{\xi^5}+(f_4)_{\xi^4}-(f_3)_{\xi^3}+(f_2)_{\xi^2}-(f_1)_{\xi}\big)+f_0,\\
        P_1&=-3(f_6)_{\xi^4}+\frac{5}{2}(f_5)_{\xi^3}-2(f_4)_{\xi^2}+\frac{3}{2}(f_3)_{\xi}-f_2,\\
        P_2&=\frac{9}{2}(f_6)_{\xi^2}-\frac{5}{2}(f_5)_{\xi}+f_4,\\
        P_3&=-f_6.
    \end{align*}
    \end{small}%
    Obviously, we arrive at
    \begin{prop}
      If $P_0,P_1,P_2$ and $P_3$ all are non-negative, then the extremal curve is stable maximal. Otherwise, the extremal curve is unstable.
    \end{prop}
    Furthermore, in view of \eqref{extm-eq} and the representations for $f_i,\;i=0,1,\cdots,6$ right behind \eqref{sec-f}, we have
    \begin{small}
    \begin{align*}
        P_0=\;&\frac{2}{3}\varphi_{\xi^2}^2(76\varphi^2-3\epsilon-150\varphi_{\xi})+\frac{8}{9}\varphi\varphi_{\xi^2}\varphi_{\xi}(22\varphi^2+33\varphi_{\xi}+14\epsilon)+6\varphi_{\xi}^4
        \nonumber\\
        &\qquad\qquad\qquad\qquad\qquad+\frac{4}{27}(114\varphi+49\epsilon)\varphi_{\xi}^3+\frac{2}{27}(12\varphi^4+8\epsilon\varphi^2+1)\varphi_{\xi}^2,\\
        P_1=\;&-\frac{117}{2}\varphi_{\xi^2}^2-\frac{1}{3}(2\varphi^2+7\epsilon-219\varphi_{\xi})\varphi\varphi_{\xi^2}+\frac{88}{3}\varphi_{\xi}^3+\frac{5}{6}(39\varphi^2+8\epsilon)\varphi_{\xi}^2\\
        &\qquad\qquad\qquad\qquad\qquad+\frac{1}{9}(2\varphi^4+25\epsilon\varphi^2-4)\varphi_{\xi}+\frac{1}{162}(\varphi^2-4\epsilon)(2\varphi^2+\epsilon)^2,\\
        P_2=\;&33\varphi_{\xi^2}\varphi+48\varphi_{\xi}^2+(19\varphi^2-4\epsilon)\varphi_{\xi}+\frac{\varphi^4}{2}-\epsilon\varphi^2+\frac{1}{2},\\
        P_3=\;&3(\varphi^2+9\varphi_{\xi}-\epsilon).
    \end{align*}
    \end{small}%
    Now let us investigate the stability for the concrete extremal curves through calculating the values of $P_0,P_1, P_2$ and $P_3$.

    \noindent {\bf Case 1.}
    If the fully affine curvature $\varphi$ of the extremal curve is constant, that is,
    $\varphi_{\xi^i}=0, i=1,2,\cdots$,
    then
    \begin{small}
    \begin{equation*}
    P_0=0,\quad P_1=\frac{1}{162}(\varphi^2-4\epsilon)(2\varphi^2+\epsilon)^2, \quad P_2=\frac{\varphi^4}{2}-\epsilon\varphi^2+\frac{1}{2},\quad P_3=3(\varphi^2-\epsilon),
    \end{equation*}
    \end{small}%
    and
    \begin{small}
    \begin{align*}
    l''(t_0)&=-\frac{1}{2}\int^1_0\bigg(\frac{9}{2}U_{\xi^4}^2+3(\varphi^2-\epsilon)U_{\xi^3}^2+\frac{ (\varphi^2-\epsilon)^2}{2}U_{\xi^2}^2\\
    &\qquad\qquad\qquad\qquad\qquad\qquad\qquad+\frac{(\varphi^2-4\epsilon)(2\varphi^2+\epsilon)^2}{162}U_{\xi}^2\bigg)gdp.
    \end{align*}
    \end{small}%
    It immediately follows that
    \begin{prop}
     (1) The curves with constant fully affine curvature  and $\epsilon=-1$ are stable fully affine maximal curves.
     (2) The curves with constant fully affine curvature and $\epsilon=1$ and $\varphi^2\ge4$ are stable fully affine maximal curves.
    \end{prop}
    \begin{rem}
    It is obvious that the curves $y=x^\alpha ~\left(\alpha\notin\{0,1,\frac{1}{2},2\}\right)$ and $y=x\log x $ are stable maximal curves, and the logarithmic spiral with polar coordinates $\rho=\exp(\theta\tan\varphi)~(0\leq\varphi<\pi/2)$ is unstable (see \cite{ks} for a classification of plane curves with constant fully affine curvature).
    \end{rem}
    \noindent {\bf Case 2.} In \cite{ks}, we can find that the following $\varphi(\xi)$ are solutions of \eqref{extm-eq},
    \begin{equation*}
       \varphi(\xi)=\frac{3\sqrt{2}}{2}\tanh(\frac{\sqrt{2}}{3}\xi-c)\; , \; \varphi(\xi)=\frac{3\sqrt{2}}{2}\coth(\frac{\sqrt{2}}{3}\xi-c),
    \end{equation*}
    for $\epsilon=1$, and
    \begin{equation*}
       \varphi(\xi)=\frac{3\sqrt{2}}{2}\tan(c-\frac{\sqrt{2}}{3}\xi),\;  \varphi(\xi)=\frac{3\sqrt{2}}{2}\cot(\frac{\sqrt{2}}{3}\xi-c), \;   \varphi(\xi)=\frac{9}{2(\xi-c)}\pm\frac{\sqrt{2}}{2}
    \end{equation*}
    for $\epsilon=-1$. Without loss of generality, in the following, we assume $c=0$. Let us calculate stabilities of these extremal curves.

   {\bf 1.} If $\varphi(\xi)=-\frac{3\sqrt{2}}{2}\tan(\frac{\sqrt{2}}{3}\xi)$ and $\epsilon=-1$, a direct computation yields
   \begin{small}
   \begin{equation*}
      P_3=3(\varphi^2+9\varphi_{\xi}-\epsilon)=-\frac{3\left(7\cos(\frac{\sqrt{2}}{3}\xi)^2+9\right)}{2\cos(\frac{\sqrt{2}}{3}\xi)^2}<0.
   \end{equation*}
   \end{small}%
   Hence, the corresponding extremal curve is unstable.

   {\bf 2.} If $\varphi(\xi)=\frac{3\sqrt{2}}{2}\cot(\frac{\sqrt{2}}{3}\xi)$ and $\epsilon=-1$, by calculation, we get
   \begin{small}
   \begin{equation*}
      P_3=3(\varphi^2+9\varphi_{\xi}-\epsilon)=-3\left(\frac{9}{2}\cot(\frac{\sqrt{2}}{3}\xi)^2+8\right)<0.
   \end{equation*}
   \end{small}%
   Therefore, this corresponding extremal curve also is unstable.

   {\bf 3.}  If $\varphi=\frac{\sqrt{2}}{2}+\frac{9}{2\xi}$ and $\epsilon=-1$,
   we can obtain
   \begin{small}
   \begin{align*}
     P_0&=\;\frac{81\Big(4\sqrt{2}\xi({\xi}^2-81)-6{\xi}^2+2475\Big)}{4{\xi}^8},\\
     P_1&=\;\frac{27\Big(\sqrt{2}\xi(93-2{\xi}^2)+3{\xi}^2-603\Big)}{2{\xi}^6},\\
     P_2&=\;\frac{9\Big(12\sqrt{2}\xi(2{\xi}^2-43)+4{\xi}^4-36{\xi}^2+2781\Big)}{32{\xi}^4},\\
     P_3&=\;\frac{9(6\sqrt{2}\xi+2{\xi}^2-27)}{4{\xi}^2}.
   \end{align*}
   \end{small}%
   To proceed further, we need to calculate the sign of $P_0,P_1, P_2$ and $P_3$.
   \begin{itemize}
     \item $P_0\geq0$ if and only if  $\xi\geq \xi_1$, where $\xi_1=\frac{\sqrt{2}}{4}-\frac{(9250\sqrt{2}+24\sqrt{155171})^{1/3}}{4}-\frac{217}{2(9250\sqrt{2}+24\sqrt{155171})^{1/3}}\approx-10.55$;
     \item $P_1\geq0$ if and only if $\xi\leq\xi_2$, where $\xi_2=\frac{\sqrt{2}}{4}-\frac{(4450\sqrt{2}+100\sqrt{2398})^{1/3}}{4}-\frac{217}{2(4450\sqrt{2}+100\sqrt{2398})^{1/3}}\approx-8.03$;
     \item $P_2\geq0$;
     \item $P_3\geq0$ if and only if $\xi\in(-\infty, -\frac{9\sqrt{2}}{2}]\cup[\frac{3\sqrt{2}}{2},+\infty)$.
   \end{itemize}
   In fact, we can find $\xi_1<\xi_2<-\frac{9\sqrt{2}}{2}\approx-6.36$.
    Thus, when $\xi\in[\xi_1,\xi_2]$, the extremal curve with $\varphi=\frac{\sqrt{2}}{2}+\frac{9}{2\xi}$ and $\epsilon=-1$ is a stable maximal fully affine curve.

    If $\varphi=-\frac{\sqrt{2}}{2}+\frac{9}{2\xi}$ and $\epsilon=-1$,
   by a trivial computation, we have
   \begin{small}
   \begin{align*}
     P_0&=\;\frac{81\Big(-4\sqrt{2}\xi({\xi}^2-81)-6{\xi}^2+2475\Big)}{4{\xi}^8},\\
     P_1&=\;\frac{27\Big(-\sqrt{2}\xi(93-2{\xi}^2)+3{\xi}^2-603\Big)}{2{\xi}^6},\\
     P_2&=\;\frac{9\Big(-12\sqrt{2}\xi(2{\xi}^2-43)+4{\xi}^4-36{\xi}^2+2781\Big)}{32{\xi}^4},\\
     P_3&=\;\frac{9(-6\sqrt{2}\xi+2{\xi}^2-27)}{4{\xi}^2}.
   \end{align*}
   \end{small}%
   We discuss this in some detail.
   \begin{itemize}
     \item $P_0\geq0$ if and only if $\xi\leq \xi_3$, where $\xi_3=\frac{(9250\sqrt{2}+24\sqrt{155171})^{1/3}}{4}+\frac{217}{2(9250\sqrt{2}+24\sqrt{155171})^{1/3}}-\frac{\sqrt{2}}{4}\approx10.55$;
     \item $P_1\geq0$ if and only if $\xi\geq \xi_4$, where $\xi_4=\frac{(4450\sqrt{2}+100\sqrt{2398})^{1/3}}{4}+\frac{217}{2(4450\sqrt{2}+100\sqrt{2398})^{1/3}}-\frac{\sqrt{2}}{4}\approx8.03$;
     \item $P_2\geq0$;
     \item $P_3\geq0$ if and only if $\xi\in(-\infty, -\frac{3\sqrt{2}}{2}]\cup[\frac{9\sqrt{2}}{2},+\infty)$.
   \end{itemize}
    We can see $\xi_3>\xi_4>\frac{9\sqrt{2}}{2}$. Thus, when $\xi\in[\xi_4,\xi_3]$, the extremal curve with $\varphi=-\frac{\sqrt{2}}{2}+\frac{9}{2\xi}$ and $\epsilon=-1$ is a stable maximal fully affine curve.

   {\bf 4.}  If $\varphi(\xi)=\frac{3\sqrt{2}}{2}\coth(\frac{\sqrt{2}}{3}\xi)$ and $\epsilon=1$, one can verify that
   \begin{align*}
     P_0&=\;\frac{20\Big(28\cosh^4\left(\frac{\sqrt{2}\xi}{3}\right)+119\cosh^2\left(\frac{\sqrt{2}\xi}{3}\right)+18\Big)}{27\sinh^8\left(\frac{\sqrt{2}\xi}{3}\right)},\\
     P_1&=\;\frac{25\cosh^6\left(\frac{\sqrt{2}\xi}{3}\right)+15\cosh^4\left(\frac{\sqrt{2}\xi}{3}\right)-4398\cosh^2\left(\frac{\sqrt{2}\xi}{3}\right)-2878}{81\sinh^6\left(\frac{\sqrt{2}\xi}{3}\right)},\\
     P_2&=\;\frac{49\cosh^4\left(\frac{\sqrt{2}\xi}{3}\right)-96\cosh^2\left(\frac{\sqrt{2}\xi}{3}\right)+356}{8\sinh^4\left(\frac{\sqrt{2}\xi}{3}\right)},\\
     P_3&=\;\frac{3\Big(7\cosh^2\left(\frac{\sqrt{2}\xi}{3}\right)-16\Big)}{2\sinh^2\left(\frac{\sqrt{2}\xi}{3}\right)}.
   \end{align*}
   Let us present the details in a more concrete form.
   \begin{itemize}
     \item $P_0>0$;
     \item $P_1\geq0$ if and only if $\xi\in(-\infty,-\xi_5]\cup[\xi_5,+\infty)$, where
     \begin{small}
     \begin{equation*}
     \xi_5=\frac{3\sqrt{2}}{2}\text{arccosh}
     \left(\frac{\sqrt{30\,\sqrt {163}\cos \left( \left(\arccos  {\frac {185\,\sqrt {163}}{26569}} \right) /3 \right) -5}}{5}\right)\approx4.17;
     \end{equation*}
     \end{small}%
     \item $P_2\geq0$;
     \item $P_3\geq0$ if and only if
     \begin{equation*}
     \xi\in(-\infty, -\frac{3\sqrt{2}\text{arccosh}\left(\frac{4\sqrt{7}}{7}\right)}{2}]\cup[\frac{3\sqrt{2}\text{arccosh}\left(\frac{4\sqrt{7}}{7}\right)}{2},+\infty).
     \end{equation*}
   \end{itemize}
    It shows $\xi_5>\frac{3\sqrt{2}\text{arccosh}\left(\frac{4\sqrt{7}}{7}\right)}{2}\approx2.06$. Thus, when $\xi\in(-\infty,-\xi_5]\cup[\xi_5,+\infty)$, the extremal curve with $\varphi(\xi)=\frac{3\sqrt{2}}{2}\coth(\frac{\sqrt{2}}{3}\xi)$ and $\epsilon=1$ is a stable maximal fully affine curve.

   {\bf 5.}  If $\varphi(\xi)=\frac{3\sqrt{2}}{2}\tanh(\frac{\sqrt{2}}{3}\xi)$ and $\epsilon=1$,
   we can obtain
   \begin{small}
   \begin{align*}
     P_0&=\;\frac{20\Big(28\cosh^4\left(\frac{\sqrt{2}\xi}{3}\right)-175\cosh^2\left(\frac{\sqrt{2}\xi}{3}\right)+165\Big)}{27\cosh^8\left(\frac{\sqrt{2}\xi}{3}\right)},\\
     P_1&=\;\frac{25\cosh^6\left(\frac{\sqrt{2}\xi}{3}\right)-90\cosh^4\left(\frac{\sqrt{2}\xi}{3}\right)-4293\cosh^2\left(\frac{\sqrt{2}\xi}{3}\right)+7236}{81\cosh^6\left(\frac{\sqrt{2}\xi}{3}\right)},\\
     P_2&=\;\frac{49\cosh^4\left(\frac{\sqrt{2}\xi}{3}\right)-2\cosh^2\left(\frac{\sqrt{2}\xi}{3}\right)+309}{8\cosh^4\left(\frac{\sqrt{2}\xi}{3}\right)},\\
     P_3&=\;\frac{3\Big(7\cosh^2\left(\frac{\sqrt{2}\xi}{3}\right)+9\Big)}{2\cosh^2\left(\frac{\sqrt{2}\xi}{3}\right)}.
   \end{align*}
   \end{small}%
   By a direct computation, we find
   \begin{itemize}
     \item $P_0\geq0$ if and only if $\xi\in(-\infty,-\xi_6]\cup[-\xi_7,\xi_7]\cup[\xi_6,+\infty)$, where
         \begin{small}
         \begin{equation*}
         \xi_6=\frac{3\sqrt{2}}{2}\text{arccosh}\sqrt{\frac{175+\sqrt{12145}}{56}}\approx 3.08,
         \end{equation*}
         \end{small}%
         and
         \begin{small}
         \begin{equation*}
         \xi_7=\frac{3\sqrt{2}}{2}\text{arccosh}\sqrt{\frac{175-\sqrt{12145}}{56}}\approx0.82;
         \end{equation*}
         \end{small}%

     \item $P_1\geq0$ if and only if $\xi\in(-\infty,-\xi_8]\cup[-\xi_9,\xi_9]\cup[\xi_8,+\infty)$, where
     \begin{small}
     \begin{equation*}
       \xi_8=\frac{3\sqrt{2}}{2}\text{arccosh}
     \sqrt{\frac{6}{5}\left(\sqrt {163}\cos \left( \frac{\pi-\arccos \left( \frac{185\sqrt{163}}{26569}\right)}{3} \right) +1\right)}\approx4.25,
     \end{equation*}
     \end{small}%
     and
     \begin{small}
     \begin{equation*}
       \xi_9=\frac{3\sqrt{2}}{2}\text{arccosh}
     \sqrt{\frac{6}{5}\left(\sqrt {163}\cos \left( \frac{\pi+\arccos \left( \frac{185\sqrt{163}}{26569}\right)}{3} \right) +1\right)}\approx1.57;
     \end{equation*}
     \end{small}%
     \item $P_2\geq0$;
     \item $P_3>0$.
   \end{itemize}
     Thus, when $\xi\in(-\infty,-\xi_8]\cup[-\xi_7,\xi_7]\cup[\xi_8,+\infty)$, this extremal curve with $\varphi(\xi)=\frac{3\sqrt{2}}{2}\tanh(\frac{\sqrt{2}}{3}\xi)$ and $\epsilon=1$ is a stable maximal fully affine curve.

\section{Isoperimetric inequality in fully affine geometry}\label{sec-iso}

    In this section, we mainly focus on closed curves on plane. According to the assumption in the first paragraph of Section \ref{sec-Evo},
    we have $\bm{[}C_p,C_{p^2}\bm{]}\neq0$, which implies the Euclidean curvature $\kappa\neq0$.
    Thus, in this paper, a closed curve is convex.
    The isoperimetric inequality for a domain in $\R^n$ is one of the most beautiful results in geometry.
    For example, in Euclidean geometry $\R^2$, the isoperimetric inequality states, for the length $L$ of a closed curve and the area $A$ of the planar region that it encloses, that
    \begin{equation*}
      L^2\geq 4\pi A,
    \end{equation*}
    and that equality holds if and only if the curve is a circle.
    However, the area $A$ is variant under fully affine transformations,
    and then a natural question arises: What is the isoperimetric inequality in fully affine geometry.
    Firstly, we state one lemma that will help with our subsequent argument.
    \begin{lem}\label{lem-ep1} If the curve $C$ is closed, then $\displaystyle \oint_{C}\varphi d\xi=0$ and $\epsilon=1$.
    \end{lem}
    \begin{proof}
    By \eqref{affine-curv}, it is obvious to see $\displaystyle \oint_{C}\varphi d\xi=0$ for a closed curve $C$.
    Let $\sigma$ and $\mu$ be equi-affine arc length parameter and equi-affine curvature. Let $s$ and $\kappa$ be the Euclidean arc length parameter and
    Euclidean curvature. It is well known that (cf. \cite{st})
    \begin{equation*}
      \frac{d\sigma}{ds}=\kappa^{1/3}>0,\quad \mu=\kappa^{4/3}-\frac{5}{9}\kappa^{-8/3}\kappa^2_s+\frac{1}{3}\kappa^{-5/3}\kappa_{ss}.
    \end{equation*}
    Note that since $\bm{[}C_{p},~C_{p^2}\bm{]}\neq0$,  a suitable parameter choice can guarantee $\bm{[}C_{p},~C_{p^2}\bm{]}>0$.

    In equi-affine setting, it is not hard to verify
    \begin{equation*}
     \epsilon=\mathrm{sgn}(\mu),\quad \frac{d\xi}{d\sigma}=3\sqrt{\epsilon\mu}, \quad \varphi=\frac{\epsilon}{2}(\epsilon\mu)^{-3/2}\mu_{\sigma}.
    \end{equation*}
    Using integration by parts, we have
    \begin{align*}
      \oint_C\mu ds&=\;\oint_C\left(\kappa^{4/3}-\frac{5}{9}\kappa^{-8/3}\kappa^2_s+\frac{1}{3}\kappa^{-5/3}\kappa_{ss}\right) ds\\
      &=\;\oint_C\Big(\kappa^{4/3}-\frac{5}{9}\kappa^{-8/3}\kappa^2_s\Big) ds-\oint_C\Big(\frac{1}{3}\kappa_{s}\Big) d\left(\kappa^{-5/3}\right)\\
       &=\; \oint_C\kappa^{4/3}ds>0.
    \end{align*}
    Based on $\displaystyle \frac{d\xi}{d\sigma}=3\sqrt{\epsilon\mu}\neq0$, one can obtain $\mu>0$, which implies $\epsilon=1$. Hence,  we complete the proof.
    \end{proof}
    By \cite{ast,st}, we see any convex smooth embedded curve converges to an elliptical point
    when evolving according to equi-affine heat flow, that is, $\alpha=0,\;\beta=1$ in \eqref{eamotion}.
    Now we find
    \begin{lem}\label{lem-mu}
      In equi-affine setting, for any convex smooth embedded closed curve $C$,
      \begin{equation*}
      \oint_C\sqrt{\mu}d\sigma\leq2\pi,
      \end{equation*}
      and that equality holds if and only if the curve is an ellipse.
    \end{lem}
    \begin{proof}
    A direct computer shows that if $C$ is an ellipse, $\displaystyle \oint_C\sqrt{\mu}d\sigma=2\pi.$
    Under the equi-affine heat flow, \eqref{eqgtog} and \eqref{eqkt} can be represented as
    \begin{equation*}
     \frac{\bar{g}_t}{\bar{g}}=-\frac{2}{3}\mu,\quad \mu_t=\frac{4}{3}\mu^2+\frac{1}{3}\mu_{\sigma^2}.
    \end{equation*}
    In particular, by \cite{st}, if $\mu(\cdot,0)>0$, then $\mu(\cdot,t)>0$.
    Hence, if the initial curve is not an ellipse, integration by parts gives
    \begin{align*}
      \frac{d}{dt}\oint_C\sqrt{\mu}d\sigma&=\;\oint_C\left(\frac{\mu_t}{2\sqrt{\mu}}+\sqrt{\mu}\frac{\bar{g}_t}{\bar{g}}\right)d\sigma\\
      &=\;\oint_C\frac{1}{6\sqrt{\mu}}\mu_{\sigma^2}d\sigma\\
      &=\;\frac{1}{12}\oint_C\mu^{-3/2}\mu_{\sigma}^2d\sigma>0,
    \end{align*}
    which implies $\displaystyle \oint_C\sqrt{\mu}d\sigma$ is strictly monotone increasing with respect to the time $t$.
    Since any convex smooth embedded curve converges to an elliptical point
    when evolving according to equi-affine heat flow, and then we complete the proof.
    \end{proof}
    \begin{rem}
     An alternative proof of Lemma \ref{lem-mu} is to apply \eqref{su-isoi} and \eqref{st-isoi}, which imply
     \begin{equation*}
           2\oint_C\mu d\sigma\leq \left(\oint_C d\sigma\right)^2 A^{-1}\leq 8\pi^2\left(\oint_C d\sigma\right)^{-1}.
     \end{equation*}
     The Cauchy--Schwarz inequality leads to $\displaystyle \left(\oint_C\sqrt{\mu} d\sigma\right)^2\leq\oint_C d\sigma\oint_C\mu d\sigma\leq 4\pi^2$, then we achieve the desired result.
    \end{rem}

    According to $\displaystyle \oint_Cd\xi=3\oint_C\sqrt{\mu}d\sigma$, we obtain the fully affine isoperimetric inequality.
    \begin{thm}\label{thm-faii}
      In fully affine geometry, for any convex smooth embedded closed curve $C$, the fully affine perimeter of the curve
      \begin{equation*}
       L=\oint_Cd\xi\leq6\pi,
       \end{equation*}
      and that equality holds if and only if the curve is an ellipse.
    \end{thm}
    In view of $\displaystyle \oint_Cd\xi=\oint_C\sqrt{\frac{9\kappa^4+3\kappa\kappa_{ss}-5\kappa_s^2}{\kappa^2}}ds$, we have
    \begin{cor}
      In Euclidean geometry, for any convex smooth embedded closed curve $C$, we have
      \begin{equation*}
      \oint_C\sqrt{\frac{9\kappa^4+3\kappa\kappa_{ss}-5\kappa_s^2}{\kappa^2}}ds\leq6\pi,
      \end{equation*}
      and that equality holds if and only if the curve is an ellipse.
    \end{cor}

\section{Fully affine heat flow}

    The fully affine heat evolution equation is given by
    \begin{equation}\label{heat-evo}
       \frac{\partial C(p,~t)}{\partial t}=C_{\xi^2}(p,~t),  \qquad C(\cdot,~0)=C_0(\cdot).
    \end{equation}
    This implies $U=1$ and $W=0$ in \eqref{curve-motion}. Then \eqref{gtog} and \eqref{vaphi-evo} can be rewritten as
    \begin{align}
    \frac{g_t}{g}&=\frac{\epsilon}{9}(2\varphi^2\varphi_{\xi}+3\varphi\varphi_{\xi^2}+6\varphi_{\xi}^2-9\epsilon\varphi_{\xi}-9\varphi_{\xi^3}),\label{heat-gt}\\
    \varphi_t&=-\frac{2}{9}\epsilon \varphi^3\varphi_{\xi}+\frac{2}{3}\epsilon \varphi\varphi^2_{\xi}-\frac{1}{9}\varphi\varphi_{\xi}+\frac{\epsilon}{3}\varphi^2\varphi_{\xi^2}+5\epsilon \varphi_{\xi}\varphi_{\xi^2}\nonumber\\
    &\qquad\qquad\qquad\qquad\qquad\qquad-\frac{4}{3}\varphi_{\xi^2}+2\epsilon \varphi\varphi_{\xi^3}
    -3\epsilon \varphi_{\xi^4}.\label{heat-kt}
    \end{align}

 \subsection{Solitons of the heat flow}
    In general, a soliton is defined as a solution of an evolution
    equation that evolves along symmetries of the equation.
    Let us illustrate what we mean by ``evolving along symmetries''.  Taking the ordinary one-dimensional heat equation $f_t=f_{xx}$ as an example,
    we can see that the following vector fields are infinitesimal
    symmetries of this equation, that is, each one generates a one-parameter
    group of transformations which transform solutions to solutions:
    \begin{gather*}
    \begin{array}{lll}
      &{\displaystyle X_1=\frac{\partial}{\partial x}},\qquad &\text{translation in space}\vspace{0.5ex} \\
      &{\displaystyle X_2=\frac{\partial}{\partial t}},\qquad &\text{translation in time}\vspace{0.5ex} \\
      &{\displaystyle X_3=f\frac{\partial}{\partial f}},\qquad &\text{scaling in }f\vspace{0.5ex} \\
      &{\displaystyle X_4=x\frac{\partial}{\partial x}+2t\frac{\partial}{\partial t}},\qquad &\text{scaling in space and time}.\\
    \end{array}
    \end{gather*}

    Another evolution equation is CSF, $\displaystyle \frac{\partial C}{\partial t}=\kappa N$, where $C(p,t)$ is a
    one-parameter family of immersed plane curves with curvature $\kappa$ and unit normal vector $N$. It is well known that the symmetries of this flow are the rigid motions of the
    plane, translation in time, and simultaneous dilation in space and time.  Among the soliton solutions are Abresch-Langer curves, which
    evolve by rotation and dilation, and the ``Grim Reaper''  $\cos x=\exp(-y)$, which evolves by translation in the $y$-direction.

    Analogously, the symmetries of heat flow \eqref{heat-evo} are the fully affine motions of the plane, translation and simultaneous dilation in time.
    Then the soliton of heat flow \eqref{heat-evo} is a family of curves $\widehat{C}$ of the form
    \begin{equation}\label{soliton-form}
      \widehat{C}(p,t)=l(t)A(t)C(p)+H(t),
    \end{equation}
    where $I$ is an interval containing $0$, $l:I\rightarrow\R$, $A:I\rightarrow\R^{2\times2}$,  $H:I\rightarrow\R^2$ are differentiable and the determinant of $A(t)$ is $1$ such that
    $l(0)=1, H(0)=0,$ $A(0)$ is the identity matrix, and hence $\widehat{C}(p,0)=C(p)$.
    The function $l$ determines the scaling, $A$ determines the area-preserving deformation and $H$ is the translation term.

    By \eqref{affine-arc}, \eqref{affine-curv} and \eqref{soliton-form}, we find
    \begin{equation*}
     g(p,t)=g(p,0),\qquad \varphi(p,t)=\varphi(p,0), \quad \forall t
    \end{equation*}
    which implies
    \begin{equation}\label{gk0}
     \frac{\partial g}{\partial t}\equiv0,\qquad \frac{\partial\varphi}{\partial t}\equiv0.
    \end{equation}
    According to \eqref{heat-gt} and \eqref{heat-kt}, we find the  solutions for \eqref{gk0} include the case $\varphi$ is constant.
    In the following examples, we list the motions of the curves with constant fully affine curvature under the flow \eqref{heat-evo}. Throughout the
    paper the superscript ``${}^{\mathrm{T}}$'' represents the transpose of a vector or matrix.
    \begin{exmp}\label{exm-elps}
    For $\varphi\equiv0$ and $\epsilon=1$,  the initial curve $C_0$ should be an ellipse. Assume $C_0=(a_0\cos\theta, b_0\sin\theta)^{\mathrm{T}}+(x_0,y_0)^{\mathrm{T}}$.
    The solution for \eqref{heat-evo} is
    \begin{equation*}
      C(\theta,t)=\exp(-t/9)(a_0\cos\theta, b_0\sin\theta)^{\mathrm{T}}+(x_0,y_0)^{\mathrm{T}}.
    \end{equation*}
    This is a shrinking  soliton.
    \end{exmp}
    \begin{exmp} If the initial curve $C_0$ is  hyperbola $(a_0\cosh\theta, b_0\sinh\theta)^{\mathrm{T}}+(x_0,y_0)^{\mathrm{T}}$, which implies
    $\varphi\equiv0$ and $\epsilon=-1$. Then we have
    \begin{equation*}
      C(\theta,t)=\exp(t/9)(a_0\cosh\theta, b_0\sinh\theta)^{\mathrm{T}}+(x_0,y_0)^{\mathrm{T}},
    \end{equation*}
    and it is an expanding soliton.
    \end{exmp}
    \begin{exmp}Assume the initial curve $C_0=(2\theta, \theta\log(\theta))^{\mathrm{T}}+(x_0,y_0)^{\mathrm{T}}$, that is, $\varphi\equiv2$ and $\epsilon=1$. The solution for \eqref{heat-evo} is
    \begin{equation*}
      C(\theta,t)=\exp(t)(2\theta, 2t\theta+\theta\log(\theta))^{\mathrm{T}}+(x_0,y_0)^{\mathrm{T}},
    \end{equation*}
    which is an expanding soliton.
    \end{exmp}
    \begin{exmp}For $\varphi\equiv\frac{|\alpha+1|}{|\alpha-1|^{1/2}|2\alpha-1|^{1/2}}$ and $\epsilon=-\text{sgn}\Big((2\alpha-1)(\alpha-2)\Big)$, we have the solution
    \begin{small}
    \begin{equation*}
      C(\theta,t)=\left(\theta\exp\left(\frac{t}{|2\alpha^2-5\alpha+2|}\right), \theta^\alpha\exp\left(\frac{\alpha^2t}{|2\alpha^2-5\alpha+2|}\right)\right)^{\mathrm{T}}+(x_0,y_0)^{\mathrm{T}},
    \end{equation*}
    \end{small}%
    where $\displaystyle \theta>0$, $\alpha$ is a constant and $\displaystyle\alpha\notin\{0,\frac{1}{2},1, 2\}$, with the initial curve $\displaystyle C_0=\left(\theta,\theta^\alpha\right)^{\mathrm{T}}+(x_0,y_0)^{\mathrm{T}}$.
    \end{exmp}
    \begin{exmp}For $\varphi\equiv2\sin\beta \left(\beta\in(0,\pi/2)\right)$ and $\epsilon=1$,  we have the solution
    \begin{small}
    \begin{equation*}
      C(\theta,t)=\exp\left(\alpha\theta+\frac{\alpha^2-1}{\alpha^2+9}t\right)\left(\sin\left(\theta+\frac{2\alpha t}{\alpha^2+9}\right),-\cos\left(\theta+\frac{2\alpha t}{\alpha^2+9}\right)\right)^{\mathrm{T}},
    \end{equation*}
    \end{small}%
    where $\alpha=3\tan\beta$, with the initial curve $C_0=\exp\left(\alpha\theta\right)\left(\sin\theta,-\cos\theta\right)^{\mathrm{T}}$.
    \end{exmp}

    \begin{proof}[The proof of Theorem \ref{thm-slt}]
    By \cite{lt}, $C=C(\xi)$ is a local embedding of a soliton which moves away from a central point $P$ or moves along a vector $A$ at a given point in time
    if and only if
    \begin{equation}\label{eq-slt}
       \partial_tC=C_{\xi^2}=a(C-P)+f(\xi)C_{\xi}
    \end{equation}
    or
    \begin{equation}\label{eq-slt1}
       \partial_tC=C_{\xi^2}=A+f(\xi)C_{\xi}
    \end{equation}
    where $a$ is a constant and $f(\xi)C_{\xi}$ is a tangent vector field.
    If $a>0$ (or $a<0$) in \eqref{eq-slt}, then it is called expanding (or shrinking) soliton.
    It is called  translating soliton for \eqref{eq-slt1}.

    By differentiating both sides with respect to $\xi$, \eqref{eq-slt} and \eqref{eq-slt1} lead to
    \begin{equation*}
       C_{\xi^3}=aC_{\xi}+f_{\xi}C_{\xi}+fC_{\xi^2}.
    \end{equation*}
    Comparing with \eqref{c3xi}, we see that
    \begin{equation*}
     f=-\varphi,\qquad a+f_\xi=-\lambda=-\frac{2\varphi^2+3\varphi_\xi+\epsilon}{9},
    \end{equation*}
    which generate an ordinary differential equation
    \begin{equation*}
      2\varphi^2-6\varphi_\xi+9a+\epsilon=0.
    \end{equation*}

    If $\varphi$ is constant, then we have $\displaystyle \varphi^2=-\frac{9a+\epsilon}{2}$.
    Let $\displaystyle A^2=\left|\frac{9a+\epsilon}{2}\right|$.
    If $\varphi$ is not constant,
    by solving the ordinary differential equation, we have
    $\displaystyle \varphi=-\frac{3}{\xi}$,
    $\displaystyle \varphi=A\tan\left(\frac{A}{3}\xi\right)$, $\displaystyle \varphi=-A\cot\left(\frac{A}{3}\xi\right)$,
    $\displaystyle \varphi=-A\tanh\left(\frac{A}{3}\xi\right)$ or $\displaystyle \varphi=-A\coth\left(\frac{A}{3}\xi\right)$.

    On the other hand, substituting $U=1, W=\varphi$ into \eqref{gtog} and \eqref{vaphi-evo} generates
    \begin{align}
    \begin{split}\label{eq-U1Wvph}
       \frac{g_t}{g}&=\;\frac{\epsilon}{9}(-9\varphi_{\xi^3}+3\varphi\varphi_{\xi^2}+6\varphi_{\xi}^2+2\varphi^2\varphi_{\xi}),\\
       \varphi_t&=\;-3\epsilon\varphi_{\xi^4}+2\epsilon\varphi\varphi_{\xi^3}\\
       &\qquad\qquad+\frac{\epsilon\varphi^2+15\epsilon\varphi_{\xi}-4}{3}\varphi_{\xi^2}-\frac{2\varphi(\epsilon\varphi^2-3\epsilon\varphi_{\xi}-4)}{9}\varphi_{\xi}.
    \end{split}
    \end{align}
    If $g_t\equiv0$ in \eqref{eq-U1Wvph}, we obtain
    \begin{align*}
      \varphi_{\xi^3}&=\;\frac{1}{9}(3\varphi\varphi_{\xi^2}+6\varphi_{\xi}^2+2\varphi^2\varphi_{\xi}),\\
      \varphi_{\xi^4}&=\;\frac{\varphi^2+5\varphi_{\xi}}{3}\varphi_{\xi^2}+\frac{2\varphi(\varphi^2+9\varphi_{\xi})}{27}\varphi_{\xi}.
    \end{align*}
    By substituting these two equations into the above representation of $\varphi_t$ in \eqref{eq-U1Wvph}, we obtain
    \begin{equation*}
      \varphi_t=-\frac{4}{3}\left(\varphi_{\xi^2}-\frac{2}{3}\varphi\varphi_{\xi}\right).
    \end{equation*}
    Substituting $\displaystyle \varphi_{\xi^2}=\frac{2}{3}\varphi\varphi_{\xi}$ into \eqref{eq-U1Wvph} gives $g_t\equiv0$ and $\varphi_t\equiv0$.

    A direct integration shows $\displaystyle \varphi_{\xi^2}=\frac{2}{3}\varphi\varphi_{\xi}$ is equivalent to
    $\varphi$ is constant,
    $\displaystyle \varphi=-\frac{3}{\xi}$,
    $\displaystyle \varphi=A\tan\left(\frac{A}{3}\xi\right)$, $\displaystyle \varphi=-A\cot\left(\frac{A}{3}\xi\right)$,
    $\displaystyle \varphi=-A\tanh\left(\frac{A}{3}\xi\right)$ or $\displaystyle \varphi=-A\coth\left(\frac{A}{3}\xi\right)$
    for some nonzero constant $A$.

    If a soliton is closed, for $U=1$ and arbitrary $W$,   \eqref{gtog} and \eqref{vaphi-evo} produce
    \begin{align*}
    \frac{g_t}{g}&=\frac{1}{9}(2\varphi^2\varphi_{\xi}+3\varphi\varphi_{\xi^2}+6\varphi_{\xi}^2-9\varphi_{\xi}-9\varphi_{\xi^3})+W_{\xi},\\
    \varphi_t&=-\frac{2}{9} \varphi^3\varphi_{\xi}+\frac{2}{3} \varphi\varphi^2_{\xi}-\frac{1}{9}\varphi\varphi_{\xi}+\frac{1}{3}\varphi^2\varphi_{\xi^2}+5 \varphi_{\xi}\varphi_{\xi^2}\nonumber\\
    &\qquad\qquad\qquad\qquad\qquad\qquad-\frac{4}{3}\varphi_{\xi^2}+2 \varphi\varphi_{\xi^3}
    -3 \varphi_{\xi^4}+W\varphi_{\xi}.
    \end{align*}
    Then integrating $\displaystyle \frac{g_t}{g}$ on a closed curve $C$ gives
    \begin{equation*}
      \int_{C}\frac{g_t}{g}d\xi=\frac{1}{3}\int_{C}\varphi_{\xi}^2d\xi.
    \end{equation*}
    Hence, $g_t\equiv0$ and $\varphi_t\equiv0$ imply $\varphi$ is constant, that is, this closed soliton is an ellipse.
    The proof of Theorem \ref{thm-slt} is completed.
    \end{proof}

 \subsection{Energy estimates for the heat flow}
    In the following, we proceed to derive some energy inequalities for a family of closed curves $C(p,t)$ evolving according to the heat flow \eqref{heat-evo}.
    According to Lemma \ref{lem-ep1}, it is easy to see $\epsilon=1$. Then according to Section \ref{invariant-group}, by a trivial computation, we have
    \begin{equation}\label{af-ea}
    C_{\xi^2}=\frac{1}{9\mu}C_{\sigma^2}-\frac{\mu_{\sigma}}{18\mu^2}C_{\sigma},
    \end{equation}
    where $\sigma$ and $\mu$ are equi-affine arc length parameter and equi-affine curvature.
    Similarly, a direct computation yields
    \begin{equation*}
      C_{\sigma^2}=\kappa^{1/3}N-\frac{\kappa_s}{3\kappa^{5/3}}T,
    \end{equation*}
    where $s$ and $\kappa$ are Euclidean arc length parameter and Euclidean curvature.
    The tangential component of the velocity vector affects only the parametrization of
    the family of curves in the evolution, not their shape. So the existence of the family of curves
    is determined by the normal component of the velocity in Euclidean setting
    \begin{gather*}
       C_t=\frac{1}{9\kappa-5\kappa^{-3}\kappa^2_s+3\kappa^{-2}\kappa_{ss}}N,\qquad
       C(\cdot,~0)=C_0(\cdot).
    \end{gather*}
    Now we discuss the local existence and uniqueness for \eqref{heat-evo} in equi-affine setting.
    In view of \eqref{af-ea}, \eqref{heat-evo} can be rewritten as
    \begin{equation*}
       \frac{\partial C(p,~t)}{\partial t}=\frac{1}{9\mu}C_{\sigma^2}-\frac{\mu_{\sigma}}{18\mu^2}C_{\sigma},  \qquad C(\cdot,~0)=C_0(\cdot),
    \end{equation*}
    which is equivalent to choosing $\displaystyle \beta=\frac{1}{9\mu}, \quad \alpha=\frac{\sigma^2-9\beta_{\sigma}}{27}$ in \eqref{eamotion}.
    Thus, \eqref{eqgtog} and \eqref{eqkt} can be represented as
    \begin{equation*}
     \bar{g}_t=0
    \end{equation*}
    and
    \begin{align}\label{ea4order}
    \mu_t\;=&-\frac{1}{27\mu^2}\mu_{\sigma^4}+\frac{8}{27\mu^3}\mu_{\sigma}\mu_{\sigma^3}+\frac{2}{9\mu^3}\mu_{\sigma^2}^2-\frac{4}{27\mu^4}\left(9\mu_{\sigma}^2+\mu^3\right)\mu_{\sigma^2}\nonumber\\
    &\qquad\qquad\qquad\qquad\qquad+\frac{8\mu_{\sigma}^4}{9\mu^5}+\frac{2\mu_{\sigma}^2}{9\mu^2}+\frac{2}{27}(\sigma\mu_{\sigma}+2\mu).
    \end{align}
    This implies, in \eqref{ea4order}, $\displaystyle \frac{\partial}{\partial t}\frac{\partial}{\partial \sigma}=\frac{\partial}{\partial \sigma}\frac{\partial}{\partial t}$. According to Theorem \ref{lu-4order},
    there exists a unique solution of \eqref{ea4order} in some $[0,t_1), t_1>0$ where $t_1$ depends on the $H^4$-norm of $\mu_0$.
    The equi-affine curvature $\mu$ defines, up to an equi-affine transformation, a unique curve \cite{ns}.
    Hence,
    \begin{thm}[Local existence and uniqueness]\label{thm-exst}
    For the curve flow \eqref{heat-evo} with $\epsilon=1$, there exists a unique solution in an interval $[0,t_1),\;t_1>0$, where $t_1$ depends on the $H^4$-norm of the equi-affine curvature of the initial curve $C_0$.
    \end{thm}

    Assume that $P^n_m(\varphi)$ is any linear combination of terms of the type $\partial^{i_1}\varphi*\cdots*\partial^{i_m}\varphi$ with universal constant coefficients, and $n=i_1+\cdots+i_m$ is the total number of derivatives.
    With these notations, we could rewrite the evolution equation of the curvature $\varphi$ and metric $g$ as
    \begin{align*}
    \frac{g_t}{g}&=\;P^3_1(\varphi)+P_2^2(\varphi)+P_3^1(\varphi)+P_1^1(\varphi),\\
    \varphi_t&=\;-3\varphi_{\xi^5}+P_2^3(\varphi)+P_3^2(\varphi)+P_1^2(\varphi)+P_4^1(\varphi)+P_2^1(\varphi).
    \end{align*}
    Utilizing \eqref{vaphi-evo}, \eqref{varphi1t}, \eqref{varphi2t}, \eqref{varphi3t} and $\displaystyle \left(\frac{\partial^k\varphi}{\partial\xi^k}\right)_t=\left(\frac{\partial^{k-1}\varphi}{\partial\xi^{k-1}}\right)_{t\xi}-\;\frac{g_t}{g}\frac{\partial^{k}\varphi}{\partial\xi^{k}}$, we find
    \begin{lem}\label{lem-varphikt}
      The evolutions of $\varphi_{k\xi}$ are given by
      \begin{equation*}
        \left(\varphi_{k\xi}\right)_t=-3\varphi_{\xi^{k+4}}+P_2^{k+3}(\varphi)+P_3^{k+2}(\varphi)+P_1^{k+2}(\varphi)+P_4^{k+1}(\varphi)+P_2^{k+1}(\varphi).
      \end{equation*}
    \end{lem}

     Then, by \eqref{heat-gt} and \eqref{heat-kt} we have
    \begin{small}
    \begin{align*}
      \frac{d}{dt}\oint_{C}\varphi^2d\xi&=\oint_{C}\big(2\varphi\varphi_t+\varphi^2\frac{g_t}{g}\big)d\xi\\
      &=-\frac{2}{9}\oint_{C}\varphi\Big(27\varphi_{\xi^4}-\frac{27}{2}\varphi\varphi_{\xi^3}-3(\frac{3}{2}\varphi^2+15\varphi_{\xi}-4)\varphi_{\xi^2}\\
      &\qquad\qquad\qquad\qquad\qquad\qquad\qquad+(\varphi^2-9\varphi_{\xi}+\frac{11}{2})\varphi\varphi_{\xi}\Big)d\xi\\
      &=-\frac{2}{9}\oint_{C}9\varphi\varphi_{\xi}^2+27\varphi_{\xi^2}^2+9\varphi_{\xi}^3-12\varphi_{\xi}^2-9\varphi^2\varphi_{\xi}^2d\xi\\
      &=-6\oint_{C}\varphi_{\xi^2}^2d\xi-2\oint_{C}\varphi\varphi_{\xi}^2d\xi-2\oint_{C}\varphi_{\xi}^3d\xi+2\oint_{C}\varphi^2\varphi_{\xi}^2d\xi+\frac{8}{3}\oint_{C}\varphi_{\xi}^2d\xi.
    \end{align*}
    \end{small}%
    Since $\displaystyle \oint_{C}\varphi d\xi=0$,  we shall use the interpolation inequalities: For periodic function $u$ with zero mean,
    \begin{equation*}
      ||u^{(j)}||_{L^r}\leq D||u||^{1-\theta}_{L^p}\cdot||u^{(k)}||^{\theta}_{L^q},\quad \theta\in(0,1),
    \end{equation*}
    where $r,q,p,j$ and $k$ satisfy $p,q,r>1$, $j\geq0$,
    \begin{equation*}
      \frac{1}{r}=j+\theta(\frac{1}{q}-k)+(1-\theta)\frac{1}{p},
    \end{equation*}
    and
    \begin{equation*}
      \frac{j}{k}\leq\theta\leq1.
    \end{equation*}
    Here the constant $D$ depends on $r,p,q,j$ and $k$ only.
    Using this interpolation inequality, we have
    \begin{equation*}
     \Big(\oint_{C}\varphi^4d\xi\Big)^{\frac{1}{2}}\leq D_1\Big(\oint_{C}\varphi^2d\xi\Big)^{\frac{7}{8}}\Big(\oint_{C}\varphi_{\xi^2}^2d\xi\Big)^{\frac{1}{8}},
    \end{equation*}
    and
    \begin{equation*}
    \Big(\oint_{C}\varphi_{\xi}^4d\xi\Big)^{\frac{1}{2}}\leq D_2\Big(\oint_{C}\varphi^2d\xi\Big)^{\frac{3}{8}}\Big(\oint_{C}\varphi_{\xi^2}^2d\xi\Big)^{\frac{5}{8}}.
    \end{equation*}
    Therefore,
    \begin{small}
    \begin{align*}
    \oint_{C}\varphi^2\varphi_{\xi}^2d\xi&\leq\Big(\oint_{C}\varphi^4d\xi\Big)^{\frac{1}{2}}\Big(\oint_{C}\varphi_{\xi}^4d\xi\Big)^{\frac{1}{2}}\\
    &\leq D_3\Big(\oint_{C}\varphi^2d\xi\Big)^{\frac{5}{4}}\Big(\oint_{C}\varphi_{\xi^2}^2d\xi\Big)^{\frac{3}{4}}\\
    &\leq \eta_1\oint_{C}\varphi_{\xi^2}^2d\xi+\frac{27}{256}D_3^4\eta_1^{-3}\Big(\oint_{C}\varphi^2d\xi\Big)^5.
    \end{align*}
    \end{small}%
    \begin{small}
    \begin{align*}
    \left|\oint_{C}\varphi\varphi_{\xi}^2d\xi\right|&\leq\Big(\oint_{C}\varphi^2d\xi\Big)^{\frac{1}{2}}\Big(\oint_{C}\varphi_{\xi}^4d\xi\Big)^{\frac{1}{2}}\\
    &\leq D_4\Big(\oint_{C}\varphi^2d\xi\Big)^{\frac{7}{8}}\Big(\oint_{C}\varphi_{\xi^2}^2d\xi\Big)^{\frac{5}{8}}\\
    &\leq \eta_2\oint_{C}\varphi_{\xi^2}^2d\xi+\frac{3}{8}\left(\frac{5}{8}\right)^{\frac{5}{3}}D_4^{\frac{8}{3}}\eta_2^{-\frac{5}{3}}\Big(\oint_{C}\varphi^2d\xi\Big)^{\frac{7}{3}}.
    \end{align*}
    \end{small}%
    \begin{small}
    \begin{align*}
    \oint_{C}\varphi_{\xi}^2d\xi&\leq D_5\Big(\oint_{C}\varphi^2d\xi\Big)^{\frac{1}{2}}\Big(\oint_{C}\varphi_{\xi^2}^2d\xi\Big)^{\frac{1}{2}}\\
    &\leq \eta_3\oint_{C}\varphi_{\xi^2}^2d\xi+\frac{1}{4}D_5^2\eta_3^{-1}\oint_{C}\varphi^2d\xi.
    \end{align*}
    \end{small}%
    \begin{small}
    \begin{align*}
    \left|\oint_{C}\varphi_{\xi}^3d\xi\right|& \leq\Big(\oint_{C}\varphi_{\xi}^2d\xi\Big)^{\frac{1}{2}}\Big(\oint_{C}\varphi_{\xi}^4d\xi\Big)^{\frac{1}{2}}\\
    &\leq D_2D_5^{\frac{1}{2}}\Big(\oint_{C}\varphi^2d\xi\Big)^{\frac{5}{8}}\Big(\oint_{C}\varphi_{\xi^2}^2d\xi\Big)^{\frac{7}{8}}\\
    &\leq \eta_4\oint_{C}\varphi_{\xi^2}^2d\xi+\frac{7^7}{8^8}D_2^8D_5^4\eta_4^{-7}\Big(\oint_{C}\varphi^2d\xi\Big)^{5}.
    \end{align*}
    \end{small}%

    Hence, we obtain
    \begin{align*}
    \frac{d}{dt}\oint_{C}\varphi^2d\xi&\leq(2\eta_1+2\eta_2+\frac{8}{3}\eta_3+2\eta_4-6)\oint_{C}\varphi_{\xi^2}^2d\xi+p(E),
    \end{align*}
    where
    \begin{equation*}
    p(E)=\frac{27}{128}D_3^4\eta_1^{-3}E^5+\frac{3}{4}\left(\frac{5}{8\eta_2}\right)^{\frac{5}{3}}D_4^{\frac{8}{3}}E^{\frac{7}{3}}+\frac{2}{3}D_5^2\eta_3^{-1}E+\frac{1}{4}\left(\frac{7}{8\eta_4}\right)^7D_2^8D_5^4E^5,
    \end{equation*}
    and
    \begin{equation*}
    E=\oint_{C}\varphi^2d\xi.
    \end{equation*}
    By choosing $\eta_i$, $i=1,2,3,4$ so that $2\eta_1+2\eta_2+\frac{8}{3}\eta_3+2\eta_4=6$, we see
    \begin{equation}\label{1d-E}
      \frac{dE}{dt}\leq C_1(E+E^5),
    \end{equation}
    where $C_1$ is a constant.
    Using the similar proof to that of Proposition 2.5 in \cite{dks}, we may deduce
    \begin{lem}\label{lem-dks}
     Let $C:I\rightarrow\R^2$ be a smooth closed curve. For any $P^\mu_{\nu}(\varphi)$ with $\nu\geq2$ which includes only derivatives of $\varphi$ of order at most $l-1$ and $\gamma=(\mu+\frac{1}{2}\nu-1)/l<2$, one has,
       $\forall \eta>0$
    \begin{equation*}
    \oint_C|P^{\mu}_{\nu}(\varphi)|d\xi\leq\eta\oint_C|\partial^l_{\xi}\varphi|^2d\xi+c\eta^{-\frac{\gamma}{2-\gamma}}\Big(\oint_C\varphi^2d\xi\Big)^{\frac{\nu-\gamma}{2-\gamma}}+c_2\Big(\oint_C\varphi^2d\xi\Big)^{\nu+\mu-1},
    \end{equation*}
    where $c$ and $c_2$ are constant and depend only $l,\mu$ and $\nu$.
    \end{lem}

    By Lemma \ref{lem-varphikt}, it follows
    \begin{small}
    \begin{align*}
      \frac{d}{dt}\oint_{C}\left(\frac{d^n\varphi}{d\xi^n}\right)^2d\xi&=\oint_{C}2\varphi_{\xi^n}\left(\varphi_{\xi^n}\right)_t+\varphi_{\xi^n}^2\frac{g_t}{g}d\xi\\
      &=-6\oint_{C}\left(\varphi_{\xi^{n+2}}\right)^2d\xi
      +\oint_{C}\Big(P_3^{2n+3}(\varphi)+P_4^{2n+2}(\varphi)\nonumber\\
      &\qquad\qquad\qquad\qquad\qquad+P_2^{2n+2}(\varphi)+P_5^{2n+1}(\varphi)+P_3^{2n+1}(\varphi)\Big)d\xi.
    \end{align*}
    \end{small}%
    Let $l=n+2$, by Lemma \ref{lem-dks}, we have the estimates
    \begin{align*}
    \oint_C|P^{2n+3}_{3}(\varphi)|d\xi&\leq\,\eta\oint_C\left(\varphi_{\xi^{n+2}}\right)^2d\xi+(c\eta^{-4n-7}+c_2)\Big(\oint_C\varphi^2d\xi\Big)^{2n+5},\\
    \oint_C|P^{2n+2}_{4}(\varphi)|d\xi&\leq\,\eta\oint_C\left(\varphi_{\xi^{n+2}}\right)^2d\xi+(c\eta^{-2n-3}+c_2)\Big(\oint_C\varphi^2d\xi\Big)^{2n+5},\\
    \oint_C|P^{2n+2}_{2}(\varphi)|d\xi&\leq\,\eta\oint_C\left(\varphi_{\xi^{n+2}}\right)^2d\xi+c\eta^{-n-1}\oint_C\varphi^2d\xi+c_2\Big(\oint_C\varphi^2d\xi\Big)^{2n+3},\\
    \oint_C|P^{2n+1}_{5}(\varphi)|d\xi&\leq\,\eta\oint_C\left(\varphi_{\xi^{n+2}}\right)^2d\xi+(c\eta^{-\frac{4n+5}{3}}+c_2)\Big(\oint_C\varphi^2d\xi\Big)^{2n+5},\\
    \oint_C|P^{2n+1}_{3}(\varphi)|d\xi&\leq\,\eta\oint_C\left(\varphi_{\xi^{n+2}}\right)^2d\xi+c\eta^{-\frac{4n+3}{5}}\Big(\oint_C\varphi^2d\xi\Big)^{\frac{2n+9}{5}}\\
    &\qquad\qquad\qquad\qquad\qquad\qquad\quad+c_2\Big(\oint_C\varphi^2d\xi\Big)^{2n+3}.
    \end{align*}
    Using the similar procedure as the previous part, we may derive
    \begin{prop}For the fully affine heat flow, the following inequality holds.
    \begin{equation}\label{nd-E}
    \frac{d}{dt}\oint_{C}\left(\frac{d^n\varphi}{d\xi^n}\right)^2d\xi\leq D(E+E^{2n+5}),
    \end{equation}
    for some constant $D$.
    \end{prop}

    \begin{proof}[Proof of Theorem \ref{thm-toelps}]
    It is similar to the proof of Proposition A in \cite{cho}. For reader's convenience, we present the key arguments below.
    If $E(t)$ is uniformly bounded in $[0,T)$ for some $T$.
    By integrating \eqref{nd-E}, we obtain that it implies a uniform bound on the $L^2$-norm of the all derivatives of the curvature $\varphi$ with respect to the fully
    affine arc length $\xi$. Hence the local existence of Theorem \ref{thm-exst} can be employed to extend the flow beyond $T$. If we take $T$ to be $\omega$, it can be concluded that $E(t)$ must become unbounded as a finite $\omega$ is approached.

    When $\omega$ is finite and $t$ is close to $\omega$, which implies $E(t)>1$, by integrating \eqref{1d-E} from $t$ to $\omega$, we have
    \begin{equation*}
      E^4(t)\geq\frac{1}{8D}(\omega-t)^{-1}.
    \end{equation*}
     This gives the desired lower bound for the blow-up rate.

    By \eqref{heat-gt} we have
    \begin{equation*}
    \frac{dL(t)}{dt}=\frac{d}{dt}\oint_{C}\frac{g_t}{g}d\xi=\frac{1}{3}\oint_{C}\varphi_{\xi}^2d\xi>0,
    \end{equation*}
    and
    \begin{equation}\label{affine-length}
    \frac{1}{3}\int^t_0\oint_{C}\varphi_{\xi}^2d\xi dt=L(t)-L(0).
    \end{equation}
    When $\omega$ is infinity, according to Theorem \ref{thm-faii}, we have  $L(t)$ is uniformly bounded. Then it follows
    \begin{equation*}
    \int^{\infty}_0\oint_{C}\varphi_{\xi}^2(\xi,\tau)d\xi d\tau\leq L_0,
    \end{equation*}
    for some constant $L_0$.
    According to Wirtinger inequality, we have
    \begin{equation*}
    \int^{\infty}_0\oint_{C}\varphi^2(\xi,\tau)d\xi d\tau\leq L_1,
    \end{equation*}
    for some constant $L_1$.
    Hence for any $\eta>0$, there exists $j_0$ such that we can find, by the mean value theorem, $t_j\in[j,j+1]$ satisfying $E_{t_j}\leq\eta$
    for all $j\geq j_0$.
    From \eqref{1d-E} it is clear that we can find a sufficiently small $\eta$ such that $E(t)$ is less than $1$ for all $t$ in $[t_j,t_j+2]$.
    It means that $E(t)$ is uniformly bounded in $[j_0+1,\infty)$.
    It follows from \eqref{nd-E} and parabolic regularity that all spatial and time derivatives of $\varphi$ are uniformly bounded.
    In view of  \eqref{affine-length}, any sequence $\{\varphi(\cdot,t_j)\}$ contains a subsequence $\{\varphi(\cdot,t_{j_i})\}$ converging smoothly
    to a constant as $t_{j_i}\to\infty$. Since $\displaystyle \oint_{C}\varphi d\xi=0$, the constant must be zero. By Example \ref{exm-elps}, the ellipse is the static solution for heat flow \eqref{heat-evo}, and then the proof of Theorem \ref{thm-toelps} is completed.
    \end{proof}

\appendix
{\centering\section*{Appendix A}}

\setcounter{equation}{0}
\setcounter{subsection}{0}
\setcounter{thm}{0}
\renewcommand{\theequation}{A.\arabic{equation}}
\renewcommand{\thesubsection}{A.\arabic{subsection}}
\renewcommand{\thethm}{A.\thesection\arabic{thm}}

\subsection{Motions of plane curves in equi-affine setting}
    Let $C(p,t):I_1\times I_2\rightarrow\R^2$ be a family of curves where $p\in I_1\subset\R$ parameterizes each curve and $t\in I_2\subset\R$ parameterizes the family.
    According to Example \ref{exm-ea}, setting
    \begin{equation}\label{eag}
      \bar{g}(p,t)=\bm{[}C_p,C_{p^2}\bm{]}^{1/3},
    \end{equation}
    and the equi-affine arc length $\sigma$ is explicitly given by
    \begin{equation*}
      \sigma(p,t)=\int_0^p\bar{g}(p,t)dp.
    \end{equation*}
    If a plane curve $C(p)$ parametrized by equi-affine arc length $\sigma$, we have
    \begin{equation*}
      [C_{\sigma}, C_{\sigma^2}]=1.
    \end{equation*}
    By Section \ref{ea-ivs}, the equi-affine curvature is given by
    \begin{equation}\label{eak}
      \mu=\bm{[}C_{\sigma^2},C_{\sigma^3}\bm{]},
    \end{equation}
    which implies $C_{\sigma^3}=-\mu C_{\sigma}$.
    Assume that the curve $C(p,t)$ evolves according to the curve flow
    \begin{equation}\label{eamotion}
         \frac{\partial C}{\partial t}=\alpha C_{\sigma}+\beta C_{\sigma^2},
    \end{equation}
    where $\{\alpha,\beta\}$ depend only on local values of $\mu$ and its $\sigma$ derivatives.
    It is easy to check
    \begin{align*}
     C_{p}=\bar{g}C_{\sigma},\quad C_{p^2}=\bar{g}\bar{g}_{\sigma}C_{\sigma}+\bar{g}^2C_{\sigma^2},
    \end{align*}
    and
    \begin{align*}
       C_{pt}&=\;\bar{g}(\alpha_{\sigma}-\beta\mu)C_{\sigma}+\bar{g}(\alpha+\beta_{\sigma})C_{\sigma^2},\\
       C_{p^2t}&=\;\Big(\bar{g}^2(\alpha_{\sigma^2}-2\beta_{\sigma}\mu-\mu_{\sigma}\beta-\alpha\mu)+\bar{g}\bar{g}_{\sigma}(\alpha_{\sigma}-\mu\beta)\Big)C_{\sigma}\\
       &\qquad\qquad\qquad\quad+\Big(\bar{g}^2(2\alpha_{\sigma}+\beta_{\sigma^2}-\beta\mu)+\bar{g}\bar{g}_{\sigma}(\alpha+\beta_{\sigma})\Big)C_{\sigma^2}.
    \end{align*}
    In view of \eqref{eag}, we have
    \begin{equation*}
    \frac{\partial g^3}{\partial t}=\bm{[}C_{pt}, C_{p^2}\bm{]}+\bm{[}C_{p}, C_{p^2t}\bm{]},
    \end{equation*}
    which generates
    \begin{equation}\label{eqgtog}
     \frac{\bar{g}_t}{\bar{g}}=\alpha_{\sigma}-\frac{2}{3}\beta\mu+\frac{1}{3}\beta_{\sigma^2}.
    \end{equation}
    Proceed further, we have
    \begin{align*}
      C_{\sigma t}&=\;C_{t\sigma}-\frac{\bar{g}_t}{\bar{g}}C_{\sigma}\\
                  &=\;-\frac{1}{3}(\mu\beta+\beta_{\sigma^2})C_{\sigma}+(\alpha+\beta_{\sigma})C_{\sigma^2},\\
       C_{\sigma^2 t}&=\;C_{\sigma t\sigma}-\frac{\bar{g}_t}{\bar{g}}C_{\sigma^2}\\
                        &=\;-\frac{1}{3}\Big(4\mu\beta_{\sigma}+\mu_{\sigma}\beta+\alpha\mu\beta_{\sigma^3}\Big)C_{\sigma}+\frac{1}{3}(\beta_{\sigma^2}+\mu\beta)C_{\sigma^2},
    \end{align*}
    and
    \begin{align*}
      C_{\sigma^3 t}&=\;C_{\sigma^2 t\sigma}-\frac{\bar{g}_t}{\bar{g}}C_{\sigma^3}\\
      &=\;-\frac{1}{3}\Big(\beta\mu_{\sigma^2}+4\mu\beta_{\sigma^2}+3\beta\mu^2+5\beta_{\sigma}\mu_{\sigma}+3\alpha\mu_{\sigma}+\beta_{\sigma^4}\Big)C_{\sigma}\\
      &\qquad\qquad-\mu\Big(\alpha+\beta_{\sigma}\Big)C_{\sigma^2}.
    \end{align*}
    Hence,
    \begin{align}\label{eqkt}
    \mu_t&=\;\bm{[}C_{\sigma^2 t}, C_{\sigma^3}\bm{]}+\bm{[}C_{\sigma^2 }, C_{\sigma^3 t}\bm{]}\nonumber\\
         &=\;\frac{1}{3}\Big(\beta_{\sigma^4}+5\mu\beta_{\sigma^2}+5\beta_{\sigma}\mu_{\sigma}+4\beta\mu^2+\alpha\mu_{\sigma}+\mu_{\sigma^2}\beta\Big).
    \end{align}

\subsection{Local existence of a fourth-order parabolic equation}
    In \cite{gi}, the equation
    \begin{equation}\label{p^4de}
    \left\{
    \begin{aligned}
      &u_t+a(x,u,u_x)u_{xxxx}+b(x,u,u_x,u_{xx})u_{xxx}+c(x,u,u_x,u_{xx})=0,\\
      &u(x,0)=u_0(x),
    \end{aligned}
    \right.
    \end{equation}
    for $t>0$ and $x\in T=\R/(\omega \mathbb{Z})$ with $\omega>0$ was considered. For \eqref{p^4de}, assume:
    \begin{itemize}
              \item[(a)] The function $a(x,\alpha_0,\alpha_1)$ is positive;
              \item[(b)] Let $M>0$ be given. The functions $a(x,\alpha_0,\alpha_1)$, $b(x,\alpha_0,\alpha_1,\alpha_2)$ and $c(x,\alpha_0,\alpha_1,\alpha_2)$
              are smooth in their all arguments but restricted for $|\alpha_0|\leq2\mu M$ and $\omega$-periodic in $x$, where $\mu=\mu(T)>0$ denote a number in Sobolev
              inequality
              \begin{equation*}
                \|f\|_{L^{\infty}(T)}\leq\mu\|f\|_{H^1(T)} \quad \mathrm{for}\quad f\in H^1(T).
              \end{equation*}
    \end{itemize}
    \begin{thm}[Local existence for \eqref{p^4de}, \cite{gi}] \label{lu-4order}
    Let $M>0$. Assume (a) and (b). Then for any $u_0\in H^4(T)$ with $\|u_0\|_{H^4(T)}\leq M$, there is a $T_0(M)>0$ such that there exists a unique solution $u(x,t)$ of \eqref{p^4de} satisfying
    \begin{small}
    \begin{equation*}
    u\in  L^2(0,T_0(M);H^6(T)),\quad u_t\in L^2(0,T_0(M); H^2(T)), \; \|u\|_{H^4(T)}(t)\leq2M,
    \end{equation*}
    \end{small}%
    for $t\in[0, T_0(M)]$.
    \end{thm}


\end{document}